\documentclass[a4paper,fleqn]{cas-sc}

\usepackage[numbers]{natbib}

\usepackage{amsthm}
\newtheoremstyle{thmstyleone}{\topsep}{\topsep}{\itshape}{}{\bfseries}{.}{.5em}{}
\newtheoremstyle{thmstyletwo}{\topsep}{\topsep}{\normalfont}{}{\itshape}{.}{.5em}{}
\newtheoremstyle{thmstylethree}{\topsep}{\topsep}{\normalfont}{}{\bfseries}{.}{.5em}{}

\usepackage{hyperref}
\usepackage{graphicx}
\usepackage{multirow}
\usepackage{amsmath,amssymb,amsfonts}
\usepackage{xcolor}
\usepackage{booktabs}
\usepackage{algorithm}
\usepackage{algorithmicx}
\usepackage{algpseudocode}
\usepackage{subcaption}
\usepackage{array}
\newcolumntype{L}[1]{>{\raggedright\arraybackslash}m{#1}}

\newcounter{model}
\newenvironment{model}{%
  \refstepcounter{model}%
  \par\smallskip\noindent\textbf{Model~\themodel.}\par\nobreak
  \setlength{\abovedisplayskip}{2pt plus 1pt minus 1pt}%
  \setlength{\abovedisplayshortskip}{2pt plus 1pt minus 1pt}%
}{%
  \par\smallskip
}
\usepackage{cleveref}
\crefname{model}{Model}{Models}
\usepackage{tabularx}
\usepackage{makecell}

\makeatletter
\renewcommand\thead{\@ifstar{\let\tabg@pe\gape\thead@}{\let\tabg@pe\theadgape\thead@}}
\makeatother

\theoremstyle{thmstyleone}
\newtheorem{theorem}{Theorem}
\newtheorem{proposition}[theorem]{Proposition}

\theoremstyle{thmstyletwo}

\theoremstyle{thmstylethree}

\newcommand{\C}{\mathcal{C}}

\newcommand{\K}{\mathcal{K}}
\newcommand{\Pall}{\mathcal{P}^{\mathrm{all}}}
\newcommand{\Fone}{\mathcal{F}_1}
\newcommand{\Ftwo}{\mathcal{F}_2}
\newcommand{\proj}{\operatorname{proj}}

\newcommand{\confirm}[1]{\textcolor{red}{#1}}

\begin{document}
\let\WriteBookmarks\relax
\def\floatpagepagefraction{1}
\def\textpagefraction{.001}

\shorttitle{Non-Contiguous Channel Allocation in Multi-Beam Satellites}

\shortauthors{A. Gaudry and V. Mak-Hau}

\title [mode = title]{Non-Contiguous Channel Allocation in Multi-Beam Satellites: Formulations, Upper Bounds and Cooperative Matheuristics}

\author[1]{Angus Gaudry}[orcid=0000-0002-5421-8852]
\cormark[1]
\ead{agaudry@deakin.edu.au}
\credit{Conceptualization, Methodology, Software, Validation, Formal analysis, Investigation, Data curation, Visualization, Writing -- original draft}

\author[1]{Vicky Mak-Hau}[orcid=0000-0002-9306-5780]
\ead{vicky.mak@deakin.edu.au}
\credit{Methodology, Supervision, Writing -- review \& editing}

\affiliation[1]{organization={School of Information Technology, Deakin University},
            city={Waurn Ponds},
            state={VIC},
            country={Australia}}

\cortext[1]{Corresponding author}

\begin{abstract}
Demand for wireless communication is growing, and commercial satellite services are expanding with it, increasing pressure on a finite radio spectrum. Allocation efficiency directly limits how many users a satellite can serve. We study how to assign frequency sub-channels to the carriers of a multi-beam satellite, maximising the total priority of admitted carriers while keeping each within its interference tolerance. Earlier work on this interference-limited channel allocation problem placed each carrier in a single continuous frequency block. Removing this restriction can admit more carriers, but a direct integer programming model cannot solve the largest test problems in practical time, so new methods are needed. We prove the problem computationally hard, reformulate it around how carriers share sub-channels, and derive three bounds on the optimum. We develop heuristics, including matheuristics combining local search with optimisation models, and an orchestrated portfolio that schedules them, recombines their solutions and stops early once its result is proven optimal or no longer improves. On 24 problems with up to 500 carriers, the best heuristics come within about 1\% of the best known solution on average. The portfolio matches or exceeds the best heuristic for each problem on 23 of 24 with equal weights and 22 with priority weights, and is up to about 5\% better on larger, more congested problems. It requires about a quarter of the time to run every heuristic and keep the best, and under half that of its heuristics run in a fixed order, which return slightly worse solutions on average.
\end{abstract}

\begin{highlights}
\item Allowing non-contiguous channels admits up to 7\% more carriers than contiguous plans
\item Three new upper bounds prove optimality on 6 instances where carriers must be refused
\item A portfolio of heuristics shares sub-channel patterns that an exact model recombines
\item The portfolio matches or beats each heuristic run alone on 23 of 24 instances
\item The portfolio needs about a quarter of the time taken to run every heuristic
\end{highlights}

\begin{keywords}
Integer Programming \sep Matheuristics \sep Algorithm Portfolio \sep Large Neighbourhood Search \sep Satellite Communication \sep Channel allocation
\end{keywords}

\maketitle

\section{Introduction}\label{sec1}

Modern multi-beam satellite systems must allocate limited spectrum among a large number of carriers while controlling co-channel interference. Resource-allocation decisions span beam configuration, power and bandwidth allocation, channelisation, and physical-layer scheduling, and these decisions are often coupled \cite{Chen2025,Gaudry2023,Li2024}. This paper focuses on the channelisation and carrier-admission stage. We assume that beam laydown, carrier-to-beam assignment, bandwidth requirements, modulation and coding schemes, and transmit powers have already been determined. The remaining decision is which carriers can be admitted and how their required bandwidth should be placed within the shared frequency band.

The available bandwidth is discretised into a common set of sub-channels accessible to all beams through digital channelisation. Each carrier has a prescribed bandwidth requirement, an interference tolerance determined by its signal-to-noise-plus-interference ratio (SNIR) requirement, an assigned beam, and a priority weight. An admitted carrier must receive its full bandwidth requirement. Carriers assigned to the same beam cannot share sub-channels, whereas carriers on different beams may do so. Such sharing induces interference, and the interference received by a carrier accumulates over all other carriers with which it overlaps. Moreover, interference is generally asymmetric due to different antenna gains of the beams and distance from users to the radiating antenna. The optimisation objective is therefore to maximise the total priority weight of admitted carriers subject to bandwidth, intra-beam disjointness, and cumulative-interference. 

This problem differs fundamentally from the classical frequency or channel allocation problem. Traditional frequency-assignment models commonly represent interference through a conflict graph, with pairwise separation constraints imposed between incompatible transmitters \cite{Hale1980,Aardal2007,Ahmed2016}. Such a representation is insufficient when feasibility depends on the sum of interference generated by several simultaneously active carriers. Under a physical interference model, carriers can be pairwise compatible while becoming infeasible when considered together \cite{Gupta2000,Halldorsson2009}. Mathematical-programming approaches incorporating cumulative interference are less common. \citet{Kiatmanaroj2012,Kiatmanaroj2016}, for example, study satellite frequency-assignment and scheduling problems with cumulative interference, while \citet{Palpant2008} and \citet{Nicholas2019} develop optimisation and constraint-programming methods for related terrestrial problems. However, existing approaches either employ representations based on contiguous frequency intervals, simplify cumulative interference through pairwise conditions, or address objectives different from carrier admission under flexible digital channelisation. 

Two recent studies address the same underlying problem of assigning frequency sub-channels to carriers under cumulative co-channel interference, with carrier-to-beam assignments and bandwidth requirements fixed upstream. \citet{Li2025} develop the hybrid heuristic--integer programming methods ALG-MUSIC and ALG-BEAT for carrier admission and bandwidth utilisation. \citet{Gaudry2026} subsequently study the same interference-limited channel allocation problem using mathematical programming, constraint programming and FLIP-based heuristics. Both works require each carrier to occupy a contiguous interval of sub-channels, represented by a starting position and an interval length. 

This work removes this contiguity requirement, so an admitted carrier may occupy any set of sub-channels of the required cardinality. Every contiguous solution remains feasible, while fragmented allocations that cannot be represented by a single interval become available. Non-contiguous allocation can therefore admit combinations of carriers that no contiguous assignment can realise. Second, the solution representation used in~\cite{Gaudry2026} no longer applies, because a single starting position cannot describe an arbitrary subset of sub-channels. The solution methodology must consequently be redesigned, since a minor modification of the existing algorithms does not suffice.

Removing contiguity also exposes useful combinatorial structure. Since individual sub-channel identities no longer matter, a feasible allocation can be described by the number of sub-channels assigned to each occupancy pattern, where a pattern identifies the set of carriers simultaneously occupying one sub-channel. This permits a Dantzig--Wolfe-style aggregation \cite{DantzigWolfe1960}, closely related to the classical pattern-based formulation of the cutting-stock problem introduced by Gilmore and Gomory \cite{GilmoreGomory1961}. Pattern-based formulations have proved effective in large-scale packing and scheduling problems \cite{Barnhart1998,Desaulniers2005,SadykovVanderbeck2013,Muritiba2010}, but the present setting has an important distinction. A pattern need only satisfy intra-beam disjointness. Cumulative interference feasibility is determined by the complete multiset of selected patterns. A pattern that is individually admissible can therefore participate in an infeasible combination. In contrast, in many packing-with-conflicts or scheduling formulations, feasibility can be certified column by column. \citet{Bakshi2017}, for example, construct columns representing sets of links that can transmit simultaneously under cumulative interference, so each generated column is itself feasible. Here, interference remains coupled across pattern multiplicities and must be enforced in the master problem.

The satellite resource-allocation literature provides additional motivation for this decomposition. Existing work has considered joint power and bandwidth allocation using metaheuristics, mathematical programming and reinforcement learning \cite{luis2020artificial,pachler2020allocating,efrem2020,Jia2019,OrtizGomez2022}, as well as hierarchical approaches that fix one class of decisions before optimising another \cite{Abdu2021,Ramirez2022}. Aravanis et al. \citet{Aravanis2016} exploit pairwise interference structure to reduce a return-link scheduling problem to maximum weight independent set on a perfect graph, which is solvable in polynomial time. Beam laydown and resource allocation have similarly been treated through hybrid optimisation and column-generation approaches \cite{Gaudry2023,Li2024}, while beam-hopping models allocate capacity in the time domain \cite{lei2020beam,Gaudry2024}. These studies primarily allocate capacity or configure beams. They do not address the admission of carriers under a fixed upstream allocation when arbitrary non-contiguous channelisation and cumulative asymmetric interference are both present.

The resulting combinatorial optimisation problem is computationally demanding. A direct mixed-integer formulation requires variables describing carrier--sub-channel allocations together with the resulting pairwise overlaps, producing a large and weak formulation as problem size grows. At the same time, no single heuristic mechanism we have tested performs uniformly well across interference regimes. These considerations motivate a cooperative hybrid framework designed to exploit the complementary strengths of a portfolio of solution approaches, combining heuristic, integer programming, and constraint programming. Hybrid optimisation methods such as large neighbourhood search \cite{Shaw1998,PisingerRopke2010} and construct--merge--solve--adapt \cite{Blum2016} show how exact optimisation can be embedded within heuristic search, while restricted-master heuristics exploit a limited pool of generated columns without requiring full branch-and-price \cite{Joncour2010,Sadykov2019,MakHau2021}. Algorithm portfolios traditionally select or schedule among alternative methods \cite{Huberman1997,GomesSelman2001,Rice1976,Kotthoff2014}. Cooperative search instead allows them to exchange information \cite{CrainicToulouse2008,Ouelhadj2010}. We exploit this idea by having several solution generators contribute occupancy patterns to a shared pool. An exact pattern-selection model then recombines patterns contributed by different methods, potentially producing a better solution than any individual method found on its own.
This portfolio structure is particularly suitable when a fixed computation budget is available, because it maintains a feasible incumbent throughout and returns the best solution found when the budget is exhausted. 

The channelisation decision is operationally suited to an anytime approach. Reconfiguring the frequency plan requires commanding the digital channeliser and retuning the corresponding ground terminals \cite{Channelizer2018}. Hence, channelisation can be optimised over a substantially longer time horizon than physical-layer scheduling. The relevant practical objective is therefore to obtain a high-quality feasible allocation reliably within the available computation time, while retaining information on solution quality through valid upper bounds. 

The main contributions of this paper are as follows.

\begin{itemize}

    \item \textbf{Non-contiguous interference-limited carrier admission.}
    We formulate carrier admission with heterogeneous bandwidth requirements, cumulative asymmetric interference, interference budget, priority-weighted admission, and arbitrary non-contiguous sub-channel allocations. We also establish the computational complexity of the resulting problem.

    \item \textbf{Pattern-based reformulation.}
    We exploit the interchangeability of sub-channels under non-contiguous allocation to reformulate the problem in terms of occupancy-pattern multiplicities. We establish the relationship between the direct and pattern formulations and show that restricting the model to a subset of patterns still yields feasible solutions to the original problem, with cumulative interference enforced across the selected patterns in the master problem. 

    \item \textbf{Upper bounding methods.}
    We develop three practical upper-bounding devices, including a minimum-overlap maximum-tolerance bound, a beam-aggregation relaxation, and a reduced-cost bound based on the pattern formulation. These provide valid optimality information even when the complete pattern set cannot be generated.

    \item \textbf{Specialised matheuristics and cooperative recombination.}
    We develop a large-neighbourhood search tailored to the cumulative-interference structure and combine it with column-generation and existing FLIP-based procedures \cite{Gaudry2026}. The portfolio pools the patterns the generators discover and solves an exact restricted pattern-selection model to recombine them.

    \item \textbf{Computational evaluation.}
    We evaluate the proposed methods on realistic and synthetic multi-beam satellite instances under both uniform and priority-weighted objectives. The experiments assess the individual generators, the orchestration rules, the upper-bounding mechanisms and the overall portfolio relative to direct optimisation, and record which component produces each returned incumbent.

\end{itemize}

The remainder of the paper is organised as follows. Section~\ref{sec:formulation} defines the problem and develops the direct and pattern-based formulations together with the upper bounds. Sections~\ref{sec:framework}, ~\ref{sec:implementation} and~\ref{sec:generator-details} present the cooperative solution framework and its principal solution generators. Section~\ref{sec:experiments} reports the computational experiments, and the final section summarises the main conclusions and directions for future work. 

\section{Problem Formulation} \label{sec:formulation}
\subsection{Notation}

Let $\mathcal{C}$ denote the set of carriers, $\mathcal{B}$ the set of beams, and $\mathcal{K} = \{1,\dots,S\}$ the set of $S$ sub-channels shared by every beam. Each carrier $i \in \mathcal{C}$ is assigned upstream to a single beam $\beta(i) \in \mathcal{B}$, requires $\Psi_i$ sub-channels, and carries a priority weight $w_i \in \mathbb{Z}_{>0}$. Integral weights make every feasible objective value integral, which the certificates of Section~\ref{sec:repository} use. We write $\mathcal{C}_b = \{i \in \mathcal{C} : \beta(i) = b\}$ for the carriers of beam $b$.

Interference is described by a non-negative matrix $\Lambda$, where $\Lambda_{ij}$ is the interference power received by carrier $i$ from carrier $j$ per sub-channel the two hold in common. The matrix is asymmetric, $\Lambda_{ij} \neq \Lambda_{ji}$ in general, with $\Lambda_{ii} = 0$ and $\Lambda_{ij} = 0$ whenever $\beta(i) = \beta(j)$, since carriers sharing a beam never share a sub-channel. Each carrier tolerates a total received interference of at most $\Gamma_i \ge 0$, derived upstream from its transmit power and the SNIR its modulation and coding scheme requires. We assume throughout that the upstream allocation is feasible in bandwidth, that is $\sum_{i \in \mathcal{C}_b} \Psi_i \leq S$ for every $b \in \mathcal{B}$, so that beam disjointness alone never prevents a carrier from being served.

\subsection{A direct formulation}

The direct formulation is presented in \Cref{mod:direct}. Carrier allocation is described by binary variables $x_i$, equal to one if carrier $i$ is activated and zero otherwise, and $y_{ik}$, equal to one if carrier $i$ occupies sub-channel $k$ and zero otherwise. We form a set for all interfering pairs $\mathcal{I} = \{\{i,j\} : i \neq j, \; \Lambda_{ij} > 0 \text{ or } \Lambda_{ji} > 0\}$, and introduce a variable $u_{ijk} \geq 0$ for each pair $\{i,j\} \in \mathcal{I}$ and sub-channel $k$. 

\begin{model}\label{mod:direct}
\begin{align}
\max_{x,y,u} \quad & \sum_{i \in \mathcal{C}} w_i\, x_i \label{eq:obj}\\
\text{s.t.} \quad
& \sum_{k \in \mathcal{K}} y_{ik} \;=\; \Psi_i\, x_i
  & \forall i \in \mathcal{C} \label{eq:demand}\\
& \sum_{i \in \mathcal{C}_b} y_{ik} \;\le\; 1
  & \forall b \in \mathcal{B},\; k \in \mathcal{K} \label{eq:beam}\\
& u_{ijk} \;\ge\; y_{ik} + y_{jk} - 1
  & \forall \{i,j\} \in \mathcal{I},\; k \in \mathcal{K} \label{eq:mcc-lo}\\
& u_{ijk} \;\le\; y_{ik}
  & \forall \{i,j\} \in \mathcal{I},\; k \in \mathcal{K} \label{eq:mcc-i}\\
& u_{ijk} \;\le\; y_{jk}
  & \forall \{i,j\} \in \mathcal{I},\; k \in \mathcal{K} \label{eq:mcc-j}\\
& \sum_{k \in \mathcal{K}} \sum_{j \,:\, \{i,j\} \in \mathcal{I}} \Lambda_{ij}\, u_{ijk} \;\le\; \Gamma_i
  & \forall i \in \mathcal{C} \label{eq:interf}\\
& x_i,\, y_{ik} \in \{0,1\}, \quad u_{ijk} \ge 0. \notag
\end{align}
\end{model}

Constraint~\eqref{eq:demand} ensures a carrier receives its full requirement of $\Psi_i$ sub-channels or none at all, so a partially served carrier does not contribute to ~\eqref{eq:obj}. Constraint~\eqref{eq:beam} enforces beam disjointness. Constraint~\eqref{eq:interf} bounds the interference carrier $i$ receives, summed over every sub-channel it occupies and every active carrier it shares one with. Constraints~\eqref{eq:mcc-lo}--\eqref{eq:mcc-j} are the McCormick envelopes of the product $y_{ik}y_{jk}$, and force $u_{ijk} = y_{ik}y_{jk}$ at every integral $y$.

The formulation nonetheless scales poorly, for two reasons. It carries $O(|\mathcal{I}|\,S)$ auxiliary variables and three times as many constraints, and its linear relaxation is weak. At a fractional point with $y_{ik} = y_{jk} = \tfrac{1}{2}$ the envelopes permit $u_{ijk} = 0$, so the relaxation admits two carriers sharing a sub-channel at no interference cost, and enforces no interference limit. 
Both limitations motivate the reformulation that follows, which carries no bilinear terms to linearise. 

\Cref{mod:direct} imposes no adjacency requirement on the $\Psi_i$ sub-channels assigned to carrier $i$. The contiguous variant studied in \cite{Gaudry2026} introduces, for each carrier $i$, a start position $\tau_i \in \{1, \dots, S - \Psi_i + 1\}$ and replaces~\eqref{eq:demand} with $y_{ik} = x_i$ for $\tau_i \le k \le \tau_i + \Psi_i - 1$ and $y_{ik} = 0$ otherwise. Every contiguous allocation satisfies~\eqref{eq:demand}, so the variant is a restriction of \Cref{mod:direct}.

Constraint~\eqref{eq:interf} distinguishes the problem from the frequency assignment models discussed in Section~\ref{sec1} in two respects. First, it sums the interference carrier~$i$ receives over all other active carriers, so feasibility depends on the activated set as a whole and in general cannot be represented by a conflict graph. Second, $\Lambda$ is asymmetric, so the pairwise relation it induces is directed, whereas the conflict graphs of frequency assignment are undirected.

\subsection{Complexity}

We refer to the problem of solving \Cref{mod:direct} as carrier admission. It is NP-hard, even under uniform priority weights. Given an arbitrary graph $G = (V,E)$ and any band size $S \ge 1$, construct an instance with $\mathcal{C} = V$, one beam per carrier, $\Psi_i = S$, $w_i = 1$ and $\Gamma_i = 0$ for all $i$, and $\Lambda_{ij} = 1$ if $\{i,j\} \in E$ and zero otherwise. Every activated carrier then occupies the whole band, so~\eqref{eq:interf} is satisfied exactly when no two activated carriers are adjacent in $G$. The feasible activated sets are then precisely the independent sets of $G$, and the optimal value of~\eqref{eq:obj} is the maximum cardinality. Maximum independent set therefore reduces to carrier admission, which inherits its NP-hardness. The reduction is a worst-case construction, not a claim about the instances an operational system presents.

\subsection{Pattern-based reformulation}\label{subsec: pattern formulation}

A \emph{pattern} $p \subseteq \mathcal{C}$ is a beam-disjoint set of carriers, that is, one containing at most one carrier from each beam. A pattern describes the carrier occupancy of a single sub-channel. Let $\mathcal{P}$ denote a set of patterns and, for each carrier $i$, let $\mathcal{P}_i = \{p \in \mathcal{P} : i \in p\}$ be those containing $i$. For $p \in \mathcal{P}_i$ write
\[
q_{ip} \;=\; \sum_{j \in p} \Lambda_{ij},
\]
the interference carrier $i$ receives from each sub-channel carrying pattern $p$.

Because contiguity has been relaxed, sub-channels are interchangeable, and a solution is determined by how many sub-channels carry each pattern, not by which ones do. Letting $z_p \in \mathbb{Z}_+$ count the sub-channels carrying pattern $p$, and $x_i \in \{0,1\}$ indicate activation as before, the pattern selection problem is

\begin{model}\label{mod:pattern}
\begin{align}
\max_{x,z} \quad & \sum_{i \in \mathcal{C}} w_i\, x_i \label{eq:pat-obj}\\
\text{s.t.} \quad
& \sum_{p \in \mathcal{P}_i} z_p \;\ge\; \Psi_i\, x_i
  & \forall i \in \mathcal{C} \label{eq:pat-demand}\\
& \sum_{p \in \mathcal{P}} z_p \;\le\; S \label{eq:pat-slot}\\
& \sum_{p \in \mathcal{P}_i} q_{ip}\, z_p \;\le\; \Gamma_i
  & \forall i \in \mathcal{C} \label{eq:pat-interf}\\
& z_p \in \mathbb{Z}_+, \quad x_i \in \{0,1\}. \notag
\end{align}
\end{model}
Constraint~\eqref{eq:pat-demand} is an inequality, not an equality, which permits a carrier to be allocated more sub-channels than it requires. This never improves the objective, and when $\mathcal{P}$ contains every beam-disjoint pattern any such solution is dominated by one that removes the surplus occurrences. Permitting it matters when $\mathcal{P}$ is smaller, where the patterns needed to remove the surplus may be unavailable (such as at the beginning of a restricted master). 

\subsubsection{Projection Equivalence of the Direct and Pattern Formulations}

In \Cref{mod:direct}, $x_i \in \{0,1\}$ indicates whether carrier~$i$ is activated and $y_{ik} \in \{0,1\}$ whether it occupies sub-channel~$k$. Because~\eqref{eq:mcc-lo}--\eqref{eq:mcc-j} force $u_{ijk} = y_{ik}y_{jk}$ at every integral $y$, and $\Lambda_{ij} = 0$ for every pair outside $\mathcal{I}$, constraint~\eqref{eq:interf} is equivalent to
\begin{equation}
\sum_{k \in \K} \sum_{j \in \C} \Lambda_{ij}\, y_{ik} y_{jk} \le \Gamma_i \qquad \forall i \in \C. \label{eq:direct-interference}
\end{equation}
Let $\Fone$ denote the set of pairs $(x,y)$ satisfying~\eqref{eq:demand}, \eqref{eq:beam} and~\eqref{eq:direct-interference}, which is the projection of the feasible set of \Cref{mod:direct} onto $(x,y)$.

Let $\Pall$ denote the set of all patterns, including the empty pattern, and $\Pall_i = \{p \in \Pall : i \in p\}$. For any set of patterns $\mathcal{P}$, let $\Ftwo(\mathcal{P})$ denote the set of pairs $(x,z)$ feasible for \Cref{mod:pattern} over $\mathcal{P}$.

\begin{proposition}[Projection equivalence]\label{prop:correspondence}
Assume $\Lambda_{ij}\ge 0$ for all $i,j\in\C$, $\Lambda_{ii}=0$,
$\Gamma_i\ge 0$ for all $i\in\C$, and $\K=\{1,\ldots,S\}$. Then 
\begin{equation}
        \proj_x(\Fone)
    =
    \proj_x\bigl(\Ftwo(\Pall)\bigr).
\end{equation} 
As both formulations have the same objective function (
$
   \max \sum_{i\in\C} w_i x_i
$), 
they have the same optimal objective value.
\end{proposition}

\begin{proof}
Given any arbitrary feasible solution $(x,y)$ of \Cref{mod:direct}. For each
sub-channel $k\in\K$, we define 
    $p_k=\{i\in\C:y_{ik}=1\}$.
Beam-disjointness in \Cref{mod:direct} implies that $p_k\in\Pall$. Let $z_p$ be
the number of sub-channels whose occupancy pattern is $p$. Then, we have that 
\begin{equation}
    \sum_{p\in\Pall} z_p\le S, 
\end{equation}
and that 
\begin{equation}
        \sum_{p\in\Pall_i} z_p
    =
    \sum_{k\in\K} y_{ik}
    =
    \Psi_i x_i, \quad \forall i \in \C
\end{equation} 
This gives 
\begin{equation}
      \sum_{p\in\Pall_i} q_{ip}z_p
    =
    \sum_{k\in\K}\sum_{j\in\C}
    \Lambda_{ij}y_{ik}y_{jk}
    \le \Gamma_i,   \quad \forall i \in \C
\end{equation} 
Hence $(x,z)$ is feasible for \Cref{mod:pattern}, with the same activation vector
$x$.

Conversely, let $(x,z)$ be feasible for \Cref{mod:pattern} over $\Pall$. Since
$\sum_{p\in\Pall} z_p\le S$, assign $z_p$ distinct sub-channels to each
pattern $p$. This yields a beam-disjoint allocation in which carrier
$i$ occupies
\[
    T_i=\sum_{p\in\Pall_i} z_p\ge \Psi_i x_i
\]
sub-channels. If $x_i=1$, retain any $\Psi_i$ of those sub-channels and
delete the remaining occurrences of carrier $i$. If $x_i=0$, delete all
of its occurrences. Let $y$ denote the resulting allocation. Then
\[
    \sum_{k\in\K} y_{ik}=\Psi_i x_i
    \qquad \forall i\in\C.
\]
Deleting carrier occurrences cannot create a beam conflict. It also
cannot increase interference, because all $\Lambda_{ij}$ are
nonnegative. Therefore, we have that
\begin{equation}
        \sum_{k\in\K}\sum_{j\in\C}
    \Lambda_{ij}y_{ik}y_{jk}
    \le
    \sum_{p\in\Pall_i} q_{ip}z_p
    \le \Gamma_i, 
    \quad \forall i\in\C.
\end{equation} 
Thus the resulting allocation is feasible for \Cref{mod:direct} and has the same
activation vector $x$.
The two feasible sets (of \Cref{mod:direct} and \Cref{mod:pattern} respectively) have the same projection onto the $x$-space. Since both objective functions depend only on $x$ and are identical, the two formulations produce the same optimal objective value.
\end{proof}

\subsection{Relaxations and Upper Bounds}\label{sec:bounds}

We develop three methods for computing an upper bound on the optimal objective value, each valid independently of the others. The minimum-overlap maximum-tolerance bound requires no optimisation solver and is available almost immediately. The beam-aggregation and reduced-cost bounds are computed by solving linear and mixed-integer programs, at greater cost and usually, though not always, with a tighter result. No method dominates across the instances tested (Table~\ref{tab:bound-source}), so the probe computes all three as time allows and stores the tightest bound obtained (Section~\ref{sec:probe}). The gap between this bound and the incumbent serves as both the optimality measure and the stopping criterion (Section~\ref{sec:repository}).

\subsubsection{Minimum-Overlap Maximum-Tolerance Bound}

We test whether $k$ carriers could be activated under optimistic assumptions, choosing their bandwidth demands, the carrier pairs that absorb their overlap, and their interference tolerances independently over all carriers. A failed test therefore shows that no set of $k$ carriers is feasible.

Let $\widetilde{\Psi}_1 \le \cdots \le \widetilde{\Psi}_{|\mathcal{C}|}$ denote the bandwidth demands in non-decreasing order. The minimum total demand of any $k$ active carriers is
\[
N_k = \sum_{\ell=1}^{k} \widetilde{\Psi}_\ell.
\]
We define an overlap unit as a pair of carriers sharing one sub-channel, and let
\[
a_k = \left\lfloor \frac{N_k}{S} \right\rfloor, \qquad r_k = N_k - a_k S.
\]
A sub-channel occupied by $m$ carriers contributes $\binom{m}{2}$ overlap units. Since $\binom{m}{2}$ is convex in $m$, the total is minimised by spreading the $N_k$ carrier--sub-channel allocations as evenly as possible over the $S$ sub-channels. This places $a_k+1$ carriers on $r_k$ sub-channels and $a_k$ on the remaining $S-r_k$, so the minimum number of overlap units is
\[
R_k = r_k \binom{a_k+1}{2} + (S - r_k) \binom{a_k}{2},
\]
which is zero whenever $N_k \le S$.

For a cross-beam carrier pair $\{i,j\}$, define the interference cost of one overlap unit between them and the largest overlap the pair can have as
\[
Q_{ij} = \Lambda_{ij} + \Lambda_{ji}, \qquad \kappa_{ij} = \min\{\Psi_i, \Psi_j\}.
\]
Charging the $R_k$ overlap units to cross-beam carrier pairs in non-decreasing order of $Q_{ij}$, with each pair $\{i,j\}$ taking at most $\kappa_{ij}$ units, gives a lower bound $\operatorname{Cost}(R_k)$ on the total interference among any $k$ active carriers.

Let $\widetilde{\Gamma}_1 \ge \cdots \ge \widetilde{\Gamma}_{|\mathcal{C}|}$ denote the interference tolerances in non-increasing order. If
\[
\operatorname{Cost}(R_k) > \sum_{\ell=1}^{k} \widetilde{\Gamma}_\ell,
\]
then no feasible solution activates $k$ carriers. Let $\bar{k}$ be the largest $k$ that passes this test. Every feasible solution activates at most $\bar{k}$ carriers, so the objective is at most the sum of the $\bar{k}$ largest carrier weights.

\subsubsection{Beam-Aggregation Bound}
\label{sec:beam-aggregation-bound}

The second bound rests on the overlap forced between pairs of beams. For any beam $b \in \mathcal{B}$, let
\[
N_b = \sum_{i \in \mathcal{C}_b} \Psi_i x_i
\]
denote its total occupied bandwidth. Since carriers within the same beam occupy disjoint sub-channels, $N_b$ is exactly the number of sub-channels occupied by beam~$b$. The occupied sub-channel sets of any two distinct beams $b$ and $c$ therefore overlap on at least
\[
\max\{0, N_b + N_c - S\}
\]
sub-channels. The bound is obtained by relaxing \Cref{mod:direct} and charging this overlap at its least possible interference cost.

For a carrier $i \in \mathcal{C}_b$ and a foreign beam $c \neq b$, let $O_{ic}$ denote the number of sub-channels carrier~$i$ shares with beam~$c$, and let
\[
\mu_{ic} = \min\left\{\Psi_i,\, \sum_{j \in \mathcal{C}_c} \Psi_j\right\}
\]
be the largest possible value of $O_{ic}$. For $0 \le o \le \mu_{ic}$, define
\[
\psi_{ic}(o) = \min\left\{ \sum_{j \in \mathcal{C}_c} \Lambda_{ij}\,\phi_j : \sum_{j \in \mathcal{C}_c} \phi_j = o,\; 0 \le \phi_j \le \Psi_j \right\}.
\]
In any allocation with $O_{ic} = o$, the numbers of sub-channels carrier~$i$ shares with each carrier of beam~$c$ form a feasible $\phi$, so carrier~$i$ receives at least $\psi_{ic}(o)$ interference from beam~$c$. The minimum is attained by filling $\phi_j$ up to $\Psi_j$ in non-decreasing order of $\Lambda_{ij}$, so $\psi_{ic}$ is convex and piecewise linear, with slopes $\Lambda_{ij}$ in that order.

Let $\mathcal{A} \subseteq \{(b,c) \in \mathcal{B} \times \mathcal{B} : b \neq c\}$ be a set of ordered beam pairs, closed under reversal, over which \Cref{mod:beam-agg} is stated.

\begin{model}\label{mod:beam-agg}
\begin{align}
\max_{x,N,O,t} \quad & \sum_{i \in \mathcal{C}} w_i\, x_i \label{eq:agg-obj}\\
\text{s.t.} \quad
& N_b = \sum_{i \in \mathcal{C}_b} \Psi_i x_i && \forall b \in \mathcal{B}, \label{eq:agg-demand}\\
& N_b \le S && \forall b \in \mathcal{B}, \label{eq:agg-cap}\\
& \sum_{i \in \mathcal{C}_b} O_{ic} \ge N_b + N_c - S && \forall (b,c) \in \mathcal{A}, \label{eq:agg-forced}\\
& 0 \le O_{ic} \le \mu_{ic} x_i && \forall i \in \mathcal{C}_b,\; (b,c) \in \mathcal{A}, \label{eq:agg-overlap}\\
& t_{ic} \ge \psi_{ic}(O_{ic}) && \forall i \in \mathcal{C}_b,\; (b,c) \in \mathcal{A}, \label{eq:agg-cost}\\
& \sum_{c : (b,c) \in \mathcal{A}} t_{ic} \le \Gamma_i x_i && \forall i \in \mathcal{C}_b,\; b \in \mathcal{B}, \label{eq:agg-budget}\\
& 0 \le x_i \le 1 && \forall i \in \mathcal{C}. \label{eq:agg-y}
\end{align}
\end{model}

Constraint~\eqref{eq:agg-forced} imposes the forced overlap $N_b + N_c - S$ for each pair in $\mathcal{A}$. Constraints~\eqref{eq:agg-overlap}--\eqref{eq:agg-cost} allow the overlap to be split among the carriers of beam~$b$, with each share charged through $\psi_{ic}$, and~\eqref{eq:agg-budget} bounds the total charge to each carrier by its interference tolerance. Because $\psi_{ic}$ is convex, it equals the maximum of its linear pieces, so~\eqref{eq:agg-cost} holds exactly when $t_{ic}$ is at least every piece. Replacing~\eqref{eq:agg-cost} with one linear inequality per piece therefore makes \Cref{mod:beam-agg} a linear program.

For any $\mathcal{A}$, \Cref{mod:beam-agg} relaxes \Cref{mod:direct}. Given a feasible solution of \Cref{mod:direct}, take each $O_{ic}$ as defined above and $t_{ic}$ as the interference carrier~$i$ receives from beam~$c$. These values satisfy every constraint of \Cref{mod:beam-agg}, since beams~$b$ and~$c$ overlap on at least $N_b + N_c - S$ sub-channels, $t_{ic} \ge \psi_{ic}(O_{ic})$, and the $t_{ic}$ sum to at most $\Gamma_i x_i$. The optimal value of \Cref{mod:beam-agg} is therefore an upper bound on that of \Cref{mod:direct}, and since adding pairs to $\mathcal{A}$ only adds constraints, a larger $\mathcal{A}$ never gives a looser bound.

The number of beam pairs grows quadratically with $|\mathcal{B}|$, so $\mathcal{A}$ is built by row generation. Only pairs with $\sum_{i \in \mathcal{C}_b} \Psi_i + \sum_{j \in \mathcal{C}_c} \Psi_j > S$ can force overlap, and those with the largest estimated cost of forced overlap form the initial set. After each solve, an omitted pair is violated when $N_b + N_c > S$ at the current solution, and the most costly violated pairs are added. Every intermediate solve gives a valid bound, and row generation converges when no omitted pair is violated. The integral form restores $x_i \in \{0,1\}$ over the converged pair set. It remains a relaxation of \Cref{mod:direct} and is never looser than the linear form, and the dual bound of the solver is valid even when the solve stops early.

\subsubsection{Reduced-Cost Bound}
\label{sec:reduced-cost-bound}

The third bound is derived from the LP relaxation of \Cref{mod:pattern}. Let $\mathcal{P}'$ denote the current restricted set of patterns, and let $\mathrm{LP}(\mathcal{P}')$ be the optimal value of the corresponding restricted master problem. Since $\mathcal{P}'$ may omit improving patterns, $\mathrm{LP}(\mathcal{P}')$ can fall below the LP optimum over $\Pall$, so a correction is needed to obtain an upper bound.

Let $\pi_i$, $\lambda_i$, and $\sigma$ denote the dual variables of~\eqref{eq:pat-demand}, \eqref{eq:pat-interf}, and~\eqref{eq:pat-slot}, respectively. For any pattern $p$, define
\[
\operatorname{val}(p) = \sum_{i \in p} \pi_i - \sum_{i,j \in p} \lambda_i \Lambda_{ij},
\qquad
\rho = \max_{p \in \Pall} \operatorname{val}(p).
\]
A pattern has positive reduced cost precisely when $\operatorname{val}(p) > \sigma$.

With exact pricing, $\rho \le \sigma$ would certify that no improving pattern has been omitted. The probe's column generation prices heuristically, so its termination alone does not provide this certificate. Instead, let $\rho_{ub}$ be any upper bound on $\rho$. Replacing $\sigma$ in the restricted-master dual solution by
\[
\sigma' = \max\{\sigma, \rho_{ub}\}
\]
gives a dual-feasible solution over $\Pall$, so by weak duality~\cite{lubbecke2005}
\[
v_{\mathrm{rc}} = \mathrm{LP}(\mathcal{P}') + S \max\{0, \rho_{ub} - \sigma\}
\]
is an upper bound on the LP optimum over $\Pall$, and hence on the optimal value of \Cref{mod:direct}.

An upper bound $\rho_{ub}$ can be obtained cheaply by dropping the non-negative interference penalties from $\operatorname{val}(p)$. Since a pattern contains at most one carrier from each beam,
\[
\rho \le \sum_{b \in \mathcal{B}} \max\left\{0, \max_{i \in \mathcal{C}_b} \pi_i\right\}.
\]
This value of $\rho_{ub}$ is available immediately. When time remains, the maximisation defining $\rho$ is also solved as a mixed-integer program. The solver's dual bound is an upper bound on $\rho$ whether or not the solve finishes, and equals $\rho$ to within the solver's optimality tolerance when it does. Since $v_{\mathrm{rc}}$ is non-decreasing in $\rho_{ub}$, using the smaller of the two values never loosens the bound.

\section{Solution Framework}
\label{sec:framework}

\subsection{Design rationale and overall algorithm}
\label{sec:framework-overview}

The proposed method is a cooperative portfolio in which heterogeneous solution generators exchange occupancy patterns as well as complete solutions. Each generator searches for feasible carrier allocations using its own representation and search mechanism. The patterns extracted from its returned allocation are added to a shared pool, irrespective of whether that allocation improves the incumbent. A restricted integer pattern-selection model then recombines the accumulated patterns, potentially producing an allocation that no individual generator found. The resulting incumbent is available to guide subsequent generator calls. Thus, the portfolio combines search across several methods with joint optimisation over the solution components they discover.

This cooperation takes advantage of the interchangeability of sub-channels under non-contiguous allocation. A pattern specifies the carriers occupying one sub-channel, but does not need a particular sub-channel index. Patterns obtained from different allocations can therefore be placed in a common pool and assigned new multiplicities during recombination. However, cumulative interference must still be enforced across all selected pattern multiplicities. The pattern-selection model \Cref{mod:pattern} provides this joint feasibility mechanism.

Alongside the incumbent and the pattern pool, the portfolio stores the tightest upper bound obtained at initialisation by the procedures of Section~\ref{sec:bounds}. These three items form the shared repository through which generation and assembly interact (Section~\ref{sec:repository}).

The search has three principal stages, as illustrated in Figure \ref{fig:portfolio-cycle}. 

\begin{figure}[pos=tbp]
\centering
\includegraphics[width=0.6\linewidth,height=0.5\textheight,keepaspectratio]{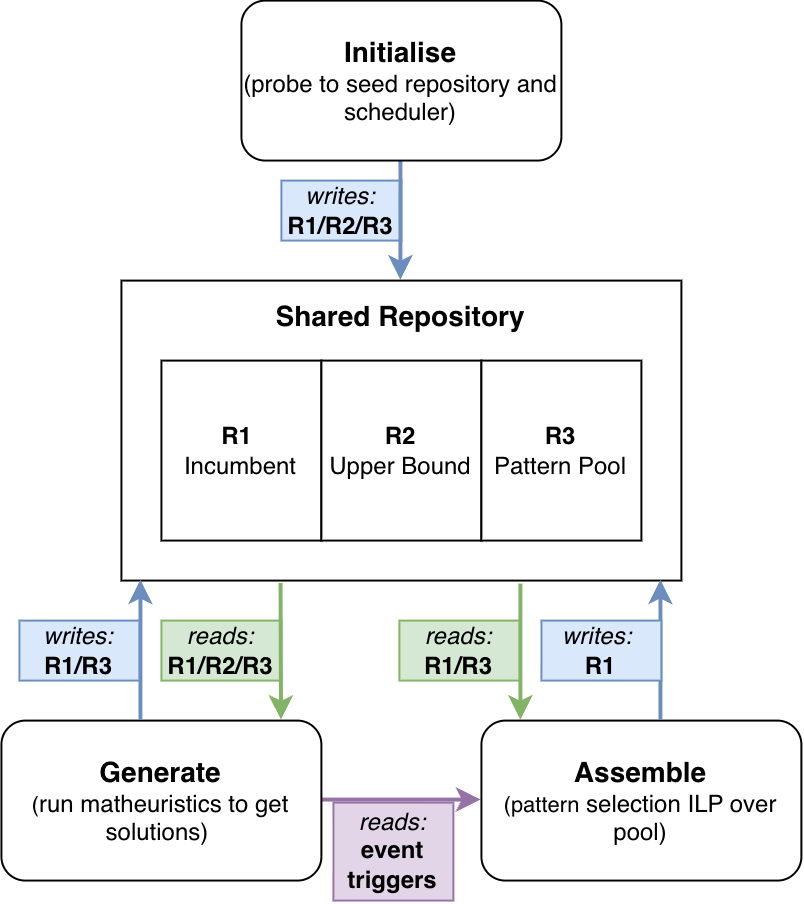}
\caption{Portfolio interactions}
\label{fig:portfolio-cycle}
\end{figure}

\begin{enumerate}
    \item \textbf{Initialisation.}
    A fast construction produces an initial feasible solution and hence an
    initial pattern pool. The upper-bounding procedures of
    Section~\ref{sec:bounds} are also invoked to provide information on the
    remaining optimality gap.

    \item \textbf{Generation.}
    A collection of complementary generators searches for improved
    feasible allocations. Any improvement replaces the incumbent, and the
    occupancy patterns of every returned allocation are added to the
    shared pool.

    \item \textbf{Assembly.}
    The pattern formulation of Section~\ref{subsec: pattern formulation} is solved over the
    accumulated and filtered pool. This allows patterns discovered independently by
    different methods to be combined into a new feasible solution.
\end{enumerate}

The incumbent is monotonically non-decreasing in objective value. Search terminates when the computational budget is exhausted, when no further improvement is obtained under the stopping rules described in Section~\ref{sec:coordinate}, or when the gap falls below one and optimality is certified. 

\subsection{Shared information and generator interface}
\label{sec:repository} 

The shared repository contains a feasible incumbent, a pattern pool, and a valid upper bound. For a feasible allocation $s=(x,y)$, define
\begin{equation}
 p_k(s)=\{i\in\mathcal C:y_{ik}=1\},
 \qquad
 \mathcal P(s)=\{p_k(s):k=1,\ldots,S\}.
 \label{eq:portfolio-extraction}
\end{equation}
By Proposition~\ref{prop:correspondence}, counting the sub-channels with each occupancy pattern gives a feasible pattern-model representation of $s$ with the same activation vector and objective value. Extraction only scans the returned allocation and never enumerates other patterns. The shared pool stores distinct pattern types, whereas recombination determines how often each type is used in the new allocation. 

\paragraph{Incumbent acceptance and pattern admission.}
Let $s^{\mathrm I}$ denote the incumbent and $f(s) = \sum_{i \in \mathcal{C}} w_i x_i$ the objective value of an allocation $s = (x,y)$. A returned allocation $s$ replaces $s^{\mathrm I}$ only if $f(s) > f(s^{\mathrm I})$. Its patterns enter the pool whether or not it replaces the incumbent,
\begin{equation}
\mathcal{P}^{\mathrm{pool}} \gets \mathcal{P}^{\mathrm{pool}} \cup \mathcal{P}(s),
\label{eq:portfolio-pool-update}
\end{equation}
so a lower-valued allocation can still contribute patterns to a later recombination. Assembly solves over a filtered subset of $\mathcal{P}^{\mathrm{pool}}$, but no pattern is ever deleted from the stored pool.

\paragraph{Upper bound.} 
During the probe, the repository keeps the tightest upper bound obtained from Section~\ref{sec:bounds}, starting from the trivial bound $\sum_{i\in\mathcal C}w_i$. Since weights are integral, each bound is rounded down before being stored. The resulting upper bound $U$ is fixed after the probe and is used throughout the remaining search to measure the gap and test for certified optimality. 

\subsection{Pattern-pool recombination}
\label{sec:assembly}

Recombination is the operation that assembles the generators' shared patterns into a joint allocation. It solves the integer pattern-selection model \eqref{eq:pat-obj}--\eqref{eq:pat-interf} over an available working set of patterns. Both carrier activation and pattern multiplicities are decisions of the optimisation problem. Bandwidth and cumulative-interference constraints are enforced jointly across the selected patterns. Consequently, recombination can combine components from different allocations without assuming that independently selected patterns are collectively feasible. Recombination runs once after the probe, whenever generation stalls (Section~\ref{sec:coordinate}), and once more at the end of the search unless optimality is already certified. 

\subsection{Search coordination and termination}
\label{sec:orchestration}

The scheduler determines which generators run and how the remaining time is distributed. It supports the pattern-sharing mechanism by balancing incumbent improvement, production of additional patterns, and opportunities for recombination. The scheduling decisions below use information from the initial probe and subsequent search progress, and all calls remain subject to the common repository rules of Section~\ref{sec:repository}. Section~\ref{sec:slice} lists the time allowances and thresholds.

\subsubsection{Initial probe}
\label{sec:probe}

The probe first runs FLIP+ (Section~\ref{sec:aux-generators}) to obtain a feasible incumbent and its occupancy patterns. It then computes the bounds of Section~\ref{sec:bounds} within its allowance, beginning with the minimum-overlap maximum-tolerance bound and the trivial sum-of-weights bound, and adds the patterns generated by its column generation to the pool. If the gap is below one after the diagnostic recombination, the run ends (Table~\ref{tab:gate}). The probe also records the instance's carrier count, which selects the schedule, and the fraction of carriers activated by the initial construction, which the near-saturation shortcut of Section~\ref{sec:coordinate} uses. 

\subsubsection{Diagnostic recombination and generator schedule}
\label{sec:schedule}

Recombination is then applied once to the probe's pool. Let $f_0$ and $f_1$ be the incumbent values immediately before and after this diagnostic call, and define its \emph{lift} as $f_1-f_0$. A closed gap terminates the run. Otherwise, a positive lift permits an additional column-generation phase before scheduling, whereas zero lift leads directly to the scheduler. The extension is skipped if pricing was already exhausted or insufficient time remains before the prescribed cutoff. When enabled, it resumes from the existing pool and stops when pricing is exhausted or that cutoff is reached.

Four solution generators appear in the schedules below (\Cref{sec:generators} describes these in detail).

\begin{itemize}
    \item \textbf{Carrier Shift and Activate Large Neighbourhood Search (CSA-LNS)}, an incumbent-repair large neighbourhood search.
    \item \textbf{Heuristic Column Generation (CG)}, which prices columns heuristically, falling back to exact pricing when none improves.
    \item \textbf{Targeted eviction heuristic}, which improves an existing incumbent by evicting carriers whose interference falls on nearly saturated carriers and refilling greedily.
    \item \textbf{FLIP-Constraint Program (CP)}, a multi-restart FLIP construction repaired by constraint programming.
\end{itemize}

Two further methods, FLIP+ and relax-and-repair CP, do not appear as scheduled entries. FLIP+, a deterministic FLIP construction with greedy top-up, supplies the probe's initial incumbent and the cold-start construction inside CSA-LNS (Section~\ref{sec:aux-generators}). Relax-and-repair CP is described in Section~\ref{sec:relax-repair}. Both also serve as standalone baselines against which the portfolio is compared.

\begin{table}[h]
\centering
\begin{tabular}{@{}lll@{}}
\toprule
Rule & Condition & Outcome \\
\midrule
G1 & gap $<1$ & stop; certified optimal \\
G2 & lift $\le 0$ & hand off to the scheduler \\
G3 & lift $>0$ & extend column generation, then hand off to the scheduler \\
\bottomrule
\end{tabular}
\caption{Gate rules, evaluated in order after the diagnostic assembly.}
\label{tab:gate}
\end{table}

Once handed off, the scheduler selects one of two schedules by instance size (Table~\ref{tab:schedules}). The schedules assign generators by role. CSA-LNS raises the number of active carriers. Targeted eviction and heuristic CG change which carriers are active, by ejection and refill or by column generation and pattern selection, and can raise the objective without changing the number of active carriers. FLIP-CP builds a fresh allocation without using the incumbent. On instances with at most $n_{\mathrm{cut}}$ carriers, the schedule cycles through CSA-LNS, heuristic CG and targeted eviction until the stall rules of Section~\ref{sec:coordinate} or the budget end the run. On larger instances, CSA-LNS runs repeatedly, and each CSA-LNS call that fails to raise the incumbent is followed by one reselection call, alternating targeted eviction and heuristic CG. The certificate of Section~\ref{sec:reduced-cost-bound} plays no part in this choice, and every call is subject to the rules of Section~\ref{sec:repository}.

\begin{table}[t]
\centering
\small
\renewcommand{\arraystretch}{1.15}
\begin{tabular}{@{}p{0.12\linewidth}
                  >{\raggedright\arraybackslash}p{0.18\linewidth}
                  >{\raggedright\arraybackslash}p{0.62\linewidth}@{}}
\toprule
Schedule & Instances & Generator calls \\
\midrule
Small & $|\mathcal C|\leq n_{\mathrm{cut}}$
   & Repeat CSA-LNS, heuristic CG and targeted eviction in turn. \\
Large & $|\mathcal C|>n_{\mathrm{cut}}$
   & Run CSA-LNS, then targeted eviction, then FLIP-CP only if the first CSA-LNS call did not raise the incumbent. Then repeat CSA-LNS, following each call that does not raise the incumbent with targeted eviction or heuristic CG, alternately. \\
\bottomrule
\end{tabular}
\caption{Generator schedules, selected by instance size.}
\label{tab:schedules}
\end{table}

\subsubsection{Progress, diversification, and stopping}
\label{sec:coordinate}

A generator call is unproductive when it raises neither the incumbent objective nor the LP relaxation value $\operatorname{LP}(\mathcal{P}^{\mathrm{pool}})$ of \Cref{mod:pattern} over the pool. The second measure credits a call that enlarges the pool's relaxation before it yields a better incumbent. A call can fail both tests and still add patterns that matter for the integer problem, so the classification serves only to control the search.

After $h_{\mathrm{stall}}$ consecutive unproductive calls, an assembly call is triggered. If it raises the incumbent, generation continues under the current schedule. Otherwise, later CSA-LNS calls start from a fresh construction instead of the incumbent (Section~\ref{sec:csa-entry}), and the count of unproductive calls is reset. If a triggered assembly call fails again after this switch, the search stops.

Certified termination occurs whenever $U-f(s^{\mathrm I})<1$. Since feasible objective values are integral and $U$ is a valid upper bound, no better feasible objective value can then exist. The implementation also uses a near-saturation shortcut. If the initial construction activates at least $99.5\%$ of carriers, the scheduler runs no generators and proceeds directly to the final assembly. This is a heuristic stopping condition, in particular, a high admission fraction need not imply a small remaining gap under non-uniform carrier weights.

In runs that reach the scheduler, a final assembly over the pool uses the time that remains once the schedule completes, the stall rules stop the search, or the unreserved budget runs out, and the incumbent is returned with $U$ and the remaining gap. If the gap certifies optimality during the search, the run ends without further generator or assembly calls. The diagnostic, stagnation-triggered, and final assembly calls use the same procedure (Section~\ref{sec:assembly}) and differ only in when they run and how much time they receive.

\subsection{Solution generators}
\label{sec:generators}

This subsection describes the portfolio's generators, beginning with the principal new one. CSA-LNS is presented first (Section~\ref{sec:csa-lns}), followed by heuristic CG (Section~\ref{sec:heuristic-CG}) and the auxiliary methods, FLIP+ and FLIP-CP from~\cite{Gaudry2026} and the targeted eviction heuristic (Section~\ref{sec:aux-generators}). Every generator's output passes through the greedy top-up of FLIP+ before it is returned, which can only add carriers to it.

\subsubsection{Carrier Shift and Activate Large Neighbourhood Search}
\label{sec:csa-lns}

CSA-LNS alternates between sub-channel reassignment and carrier activation, starting either from a fresh construction or from the incumbent after releasing the carriers that the pool's LP relaxation does not fully activate (Section~\ref{sec:csa-entry}). For an inactive carrier~$i$, define its \emph{headroom} $h_i$ as $\Gamma_i$ minus the smallest interference carrier~$i$ would receive from the active carriers on $\Psi_i$ sub-channels not occupied by its own beam, with $h_i = -\infty$ when fewer than $\Psi_i$ such sub-channels remain. A non-negative headroom is necessary for activating carrier~$i$ without altering the current allocation, but not sufficient, since the interference it causes must also stay within the active carriers' tolerances.

Each iteration samples activation candidates, favouring high headroom, and active carriers to reassign, favouring those that would receive the most interference from the candidates. The selected active carriers are reassigned to raise the smallest interference slack among the active carriers, and candidates whose headroom is then non-negative are inserted by a multi-start semi-greedy construction.

When iterations stop activating carriers, CSA-LNS applies a repair that solves a restriction of \Cref{mod:direct} by constraint programming over the carriers that block activations and the inactive carriers that could be activated alongside them, with the rest of the allocation fixed. Stalls that the repair does not resolve trigger a \emph{kick}, which deactivates the carriers that most obstruct inactive ones in the best solution found so far and resumes the search from there. Section~\ref{sec:csa-repair} gives the details.

\subsubsection{Heuristic column generation}
\label{sec:heuristic-CG}

Heuristic CG uses the pattern formulation directly as a primal search procedure. Starting from the single-carrier patterns, it alternates between solving the LP relaxation of the restricted master and pricing with its duals (Section~\ref{sec:reduced-cost-bound}), using a randomised heuristic and falling back to exact pricing when the heuristic finds no improving pattern. The resulting patterns, together with the incumbent's, are then passed to the integer pattern-selection model to produce a feasible allocation.

\subsubsection{Auxiliary generators}
\label{sec:aux-generators}

Three auxiliary methods complete the portfolio.

\paragraph{FLIP+.}
FLIP+ extends the FLIP construction of~\cite{Gaudry2026} in two ways. It starts from a fixed carrier, which makes the construction deterministic, and once FLIP's contiguous allocation is complete it passes the remaining inactive carriers to the non-contiguous activation procedure of CSA-LNS. FLIP+ supplies the probe's initial incumbent and the starting solution for cold-started CSA-LNS calls.

\paragraph{FLIP-CP.}
FLIP-CP is the flip--fix--solve method of~\cite{Gaudry2026}. Multiple FLIP restarts are generated, after which a constraint-programming repair is applied to the best construction. The repair returns a contiguous allocation, to which the greedy top-up may add carriers non-contiguously. Because its search works on contiguous placements, FLIP-CP is structurally different from the non-contiguous generators and adds diversity to the pattern pool.

\paragraph{Targeted eviction.}

The targeted eviction heuristic first applies the greedy top-up to the incumbent. It then repeatedly removes active carriers whose interference falls on nearly saturated carriers, refills the freed sub-channels with the same top-up, and keeps a result only if it improves the incumbent.

\section{Implementation details}\label{sec:implementation}

This section gives the settings and mechanisms behind the framework of Section~\ref{sec:framework}.

\subsection{Shared repository}\label{sec:impl-repository}

The Initialisation stage seeds the incumbent and the pool and sets $U$. During the Generation stage, each call receives the incumbent and $U$, and CSA-LNS also receives the activation vector of the LP relaxation of \Cref{mod:pattern} over the current pool, truncated to $4{,}000$ patterns and re-solved before each call. Each returned allocation updates the pool and, if it is better, the incumbent, under the rules of Section~\ref{sec:repository}. The Assembly stage reads the pool and the incumbent and writes only the incumbent.

\subsection{Probe and gate}\label{sec:impl-probe}

FLIP+ runs first, and the better of the minimum-overlap maximum-tolerance bound and the trivial bound is installed. Neither requires a solver. Two reserves are then withheld from the probe's allowance, one guaranteeing the reduced-cost bound a minimum slice and one for seeding the pool. The beam-aggregation bound runs against the allowance net of both reserves, and its integral form follows only if row generation converges (Section~\ref{sec:beam-aggregation-bound}) with time remaining. The reduced-cost bound then takes what is left. If it cannot run, a column-generation pass seeds the pool without attempting a certificate. Every bound is floored before it is compared or stored.

The diagnostic recombination follows the probe. Rule G1 ends the run if the gap is below one. Under rule G2 the scheduler starts if more than $10$ seconds remain. Under rule G3, column generation first resumes from the pool, unless the probe's heuristic pricing had already found no new column, and runs until $0.60T$ has elapsed, provided more than $2$ seconds remain before that point. The scheduler then starts if more than $10$ seconds remain, and otherwise a single assembly call uses the rest.

\subsection{Time allocation}\label{sec:slice}

Table~\ref{tab:portfolio-budget} lists the allowances and thresholds. Before the main schedule starts, at least $\tau_{\mathrm{end}}$ is reserved for the final assembly. Main-schedule calls and stagnation-triggered assembly calls draw only on the unreserved time, each up to its cap, and no call is launched once that time can no longer fund one of at least $5$ seconds.

\begin{table}[t]
\centering
\small
\renewcommand{\arraystretch}{1.15}
\begin{tabular}{@{}>{\raggedright\arraybackslash}p{0.34\linewidth}
                  >{\raggedright\arraybackslash}p{0.59\linewidth}@{}}
\toprule
Stage, generator or threshold & Allowance or value \\
\midrule
Initial probe & $30\%$ of $T$, at least $5$ seconds \\
Diagnostic recombination & $6\%$ of $T$, at least $3$ seconds \\
Column-generation extension (G3) & Until elapsed time reaches $0.60T$, only if more than $2$ seconds remain before that point \\
Scheduler start & More than $10$ seconds remaining \\
Final assembly reserve ($\tau_{\mathrm{end}}$) & At least $30$ seconds, extended to all time remaining when the generator loop stops \\
Minimum call length & $5$ seconds \\
\midrule
CSA-LNS & $45$ seconds per call \\
Targeted eviction & $45$ seconds per call \\
Heuristic CG & $20$ seconds per call \\
FLIP-CP & $30$ seconds per call \\
\midrule
Stall count ($h_{\mathrm{stall}}$) & $3$ consecutive unproductive calls (Section~\ref{sec:coordinate}) \\
Size threshold ($n_{\mathrm{cut}}$) & $100$ carriers (Table~\ref{tab:schedules}) \\
\bottomrule
\end{tabular}
\caption{Time allowances and thresholds. $T$ is the total wall-clock budget. Per-call caps are limited by the time remaining after the reservations for closing work.}
\label{tab:portfolio-budget}
\end{table}

\subsection{Assembly}\label{sec:impl-assembly}

Assembly first applies a reduced-cost filter that is provably safe. Using any dual solution of the pool's LP relaxation, it bounds the value of every solution that uses a given pattern and removes each pattern that cannot appear in a solution worth at least $f(s^{\mathrm I})+1$. If no solution can reach that value, the solve is skipped. When more than $600$ patterns remain, a second, heuristic filter keeps only the patterns whose reduced cost in the LP relaxation is zero to within a tolerance. At an LP optimum no pattern has positive reduced cost, so this keeps the patterns in the LP support together with any that tie them. The incumbent's patterns are then restored, and \Cref{mod:pattern} is solved over the result within the time the caller allows, warm-started from the incumbent. By the converse direction in the proof of Proposition~\ref{prop:correspondence}, any solution it returns gives a feasible allocation of equal value. The model is exact over its patterns, but under the time limit a returned solution need not be optimal over them.

\section{Generators}\label{sec:generator-details}

This section details the generators of Section~\ref{sec:generators}, together with relax-and-repair CP, which serves only as a standalone baseline.

\subsection{CSA-LNS}\label{sec:csa-details}

CSA-LNS works on an active set $\mathcal{C}'$ and, for each $i \in \mathcal{C}'$, the sub-channels $M_i = \{k \in \mathcal{K} : y_{ik} = 1\}$ it occupies, with headroom $h_i$ as defined in Section~\ref{sec:csa-lns}. It has four parts, Entry, the Improvement Loop, Repair and Exit.

\subsubsection{Entry}\label{sec:csa-entry}

A warm start keeps the incumbent's carriers whose value in the pool's LP relaxation exceeds $0.99$, releases the other active carriers, and refills with the greedy top-up of FLIP+ (Section~\ref{sec:flip-plus}). If no LP solution is available, the whole incumbent is kept and topped up. A cold start builds $\mathcal{C}'$ and $\{M_i\}$ with FLIP+, as the probe does. The orchestrator requests cold starts after a stagnation-triggered assembly fails to improve the incumbent (Section~\ref{sec:coordinate}).

\subsubsection{Improvement Loop}\label{sec:csa-loop}

Each pass first computes $h_i$ for every inactive carrier. If none has $h_i \ge 0$, the pass goes straight to Repair. Otherwise it samples a candidate set $\mathcal{C}'_3$ of up to half as many carriers as are inactive, drawn from those with finite headroom with weights increasing in $h_i$. It then samples a reassignment set $\mathcal{C}'_2$ of $15\%$ of the active carriers, at most $20$, with weights increasing in the interference each would receive from $\mathcal{C}'_3$. Each weight is divided by one plus the number of earlier passes of the call that selected the carrier. $\mathcal{C}'_2$ is reassigned to raise the smallest slack, $\Gamma_j$ minus received interference, over all of $\mathcal{C}'$.

The carriers of $\mathcal{C}'_3$ whose headroom is then non-negative are activated by a multi-restart semi-greedy construction. Its first restart visits them in order of weight, with ties broken by headroom, and every other restart samples from a restricted list of the top-ranked remaining candidates. A placement that would push an active carrier over its tolerance excludes the sub-channels involved and retries with the rest. The restart with the highest value is kept, with ties broken by lowest placement cost.

\begin{algorithm}[h]
\caption{CSA-LNS --- Improvement Loop, one pass}
\label{alg:csa-loop}
\begin{algorithmic}[1]
\State compute $h_i$ for every $i \in \mathcal{C} \setminus \mathcal{C}'$
\If{no $i \in \mathcal{C} \setminus \mathcal{C}'$ has $h_i \ge 0$} \Comment{stalled}
\State \textbf{go to} Repair (Section~\ref{sec:csa-repair})
\EndIf
\State sample $\mathcal{C}'_3 \subseteq \mathcal{C} \setminus \mathcal{C}'$ with finite headroom, weighted towards high $h_i$
\State sample $\mathcal{C}'_2 \subseteq \mathcal{C}'$, weighted towards interference received from $\mathcal{C}'_3$
\State reassign $\mathcal{C}'_2$ to raise the smallest slack over $\mathcal{C}'$
\State recompute $h_i$ for $i \in \mathcal{C}'_3$
\State activate $\{i \in \mathcal{C}'_3 : h_i \ge 0\}$ by multi-restart semi-greedy construction
\end{algorithmic}
\end{algorithm}

\subsubsection{Repair}\label{sec:csa-repair}

Repair is entered when six consecutive passes activate no carrier, or immediately when no inactive carrier has non-negative headroom. Its first mechanism is used at most once per call. It builds a pool around the active carriers that block the most inactive carriers, adds inactive carriers that could then be activated, and accepts each addition only while the degeneracy of the pool's induced interference graph stays at most six. \Cref{mod:direct} is then solved by constraint programming over the pool, for at most $10$ seconds, with activation and occupancy free for pool carriers and every other carrier fixed. A strictly improving result replaces the call's best solution and resets the count of stalled passes.

Otherwise a \emph{kick} follows in the same pass. Active carriers are ranked by how many inactive carriers each alone blocks, or, if none blocks alone, by its share of the blocking cost summed over the stuck carriers. Every identified carrier, up to $30\%$ of the active set, is removed from the call's best solution, skipping recently kicked carriers where others remain, and the Improvement Loop resumes from the result. CSA-LNS gives up after $15$ consecutive kicks without a new best solution, or when no carrier qualifies.

\begin{algorithm}[h]
\caption{CSA-LNS --- Repair}
\label{alg:csa-repair}
\begin{algorithmic}[1]
\If{the restricted solve is unused in this call}
\State $\mathcal{Q} \gets$ degeneracy-bounded pool of blockers and blocked carriers
\If{$\mathcal{Q} \neq \emptyset$ and the solve over $\mathcal{Q}$ strictly improves $\sum_i w_i x_i$}
\State update the best solution, reset the count of stalled passes
\State \textbf{resume} Improvement Loop (Section~\ref{sec:csa-loop})
\EndIf
\EndIf
\If{$15$ consecutive kicks have not improved the best solution}
\State \textbf{stop}
\EndIf
\State $\mathcal{D}_{\mathrm{kick}} \gets$ top-ranked blockers, at most $30\%$ of $\mathcal{C}'$
\If{$\mathcal{D}_{\mathrm{kick}} = \emptyset$}
\State \textbf{stop}
\EndIf
\State revert to the best solution and deactivate $\mathcal{D}_{\mathrm{kick}}$
\State \textbf{resume} Improvement Loop (Section~\ref{sec:csa-loop})
\end{algorithmic}
\end{algorithm}

\subsubsection{Exit}\label{sec:csa-exit}

CSA-LNS returns when its time allowance runs out, when the orchestrator's stall detector fires, when Repair stops, or when every carrier is active. Before returning, it builds a fresh pool around the best solution and solves the restricted model once more, for at most $30$ seconds and at most the time left in the allowance, with a minimum of one second. A call can therefore overrun its allowance by about one second plus the time to build the pool.

\subsection{Heuristic CG}\label{sec:cg-details}

Heuristic CG builds its own pattern set from the single-carrier patterns, independent of the shared pool. Each pricing step runs $12$ randomised greedy constructions with swap-and-drop local search and adds every distinct pattern with positive reduced cost. When none is found, the pricing problem is solved as a mixed-integer program over the carriers with $\pi_i > 0$, for at most $60$ seconds or the time left. Column generation stops when exact pricing also finds no improving pattern, when half of the call's budget is spent, or after $300$ iterations. The incumbent's patterns are then added, and \Cref{mod:pattern} is solved over the result within the remaining budget, warm-started from the incumbent.

\subsection{FLIP+}\label{sec:flip-plus}

FLIP~\cite{Gaudry2026} represents each carrier by a starting sub-channel and a contiguous run, and its only random element is the choice of starting carrier. FLIP+ always starts from the carrier with the largest ratio $\Psi_i/\Gamma_i$, which makes the construction deterministic. It then passes every carrier FLIP left inactive, once, to the semi-greedy construction of the Improvement Loop (Section~\ref{sec:csa-loop}), with FLIP's carriers fixed. That construction is not restricted to contiguous runs, so it can activate carriers FLIP cannot place.

\subsection{FLIP-CP}\label{sec:flip-cp}

FLIP-CP~\cite{Gaudry2026} restarts FLIP from random starting carriers for about $30\%$ of its call budget, at least $2$ seconds, and repairs the restart of highest weighted value once with the constraint program of~\cite{Gaudry2026}, restated in \Cref{mod:flip-cp}. Let $\tau_i$ be carrier~$i$'s starting sub-channel, so that an active carrier occupies sub-channels $\tau_i, \dots, \tau_i + \Psi_i - 1$, and let $\mathrm{o}_{ij} = \max\{0,\ \min(\tau_i+\Psi_i, \tau_j+\Psi_j) - \max(\tau_i, \tau_j)\}$ be the overlap between carriers $i$ and $j$. Let $\mathcal{F}$ be the set of carriers FLIP activated and $\bar{\tau}_i$ the starting sub-channel FLIP gave each $i \in \mathcal{F}$.

\begin{model}\label{mod:flip-cp}
\begin{align}
\max_{x,\tau} \quad & \sum_{i \in \mathcal{C}} w_i\, x_i \label{eq:flipcp-obj}\\
\text{s.t.} \quad
& x_i = 1, \quad \tau_i = \bar{\tau}_i
  & \forall i \in \mathcal{F} \label{eq:flipcp-fixed}\\
& 1 \;\le\; \tau_i \;\le\; S - \Psi_i + 1
  & \forall i \in \mathcal{C} \setminus \mathcal{F} \label{eq:flipcp-window}\\
& \texttt{NoOverlap}\big((\tau_i, \Psi_i, x_i)_{i \in \mathcal{C}_b}\big)
  & \forall b \in \mathcal{B} \label{eq:flipcp-nooverlap}\\
& x_i = 1 \;\Rightarrow\; \sum_{j \,:\, \{i,j\} \in \mathcal{I}} \Lambda_{ij}\, \mathrm{o}_{ij}\, x_j \;\le\; \Gamma_i
  & \forall i \in \mathcal{C} \label{eq:flipcp-budget}\\
& x_i \in \{0,1\}, \quad \tau_i \in \mathbb{Z}. \notag
\end{align}
\end{model}

Constraint~\eqref{eq:flipcp-fixed} keeps FLIP's allocation, so the program decides only whether each remaining carrier activates and where its contiguous run starts. Constraint~\eqref{eq:flipcp-nooverlap} forbids overlapping active carriers within a beam, and~\eqref{eq:flipcp-budget} bounds the interference each active carrier receives from the other active carriers. The implementation enforces~\eqref{eq:flipcp-nooverlap} through~\eqref{eq:flipcp-budget}, by giving same-beam pairs an interference coefficient that exceeds the receiving carrier's tolerance. The repaired allocation is contiguous, although the greedy top-up applied to every generator's output may add carriers non-contiguously. FLIP-CP reads neither the incumbent nor $U$.

\subsection{Targeted eviction heuristic}\label{sec:evict}

The heuristic first applies the greedy top-up to the incumbent and keeps the result if it improves. Each eviction round then scores every active carrier $a$ by $\sum_{j \in \mathcal{C}' \setminus \{a\}} \Lambda_{ja} |M_a \cap M_j| / \gamma_j$, where $\gamma_j$ is carrier~$j$'s remaining tolerance, so interference that lands on nearly saturated carriers counts most. A round makes up to four attempts. Each attempt draws $k$ carriers at random from the $\max\{k, 12\}$ highest-scoring ones, deactivates them, and refills with the greedy top-up, where $k$ is $15\%$ of the active carriers at invocation, rounded and bounded between $1$ and $20$. The round ends at the first attempt that improves the incumbent. The heuristic stops when a round fails, when fewer than two carriers are active, or when time runs out. When $k \ge 12$ the draw takes the whole candidate group, so the choice is random only on smaller instances.

\subsection{Relax-and-repair CP}\label{sec:relax-repair}

Relax-and-repair CP is a two-phase matheuristic. \citet{Gaudry2026} shows that an interval representation of the direct model improves tractability, and phase one extends that representation by allowing each carrier's interference tolerance to be exceeded by a fixed fraction. Phase two then repairs the violations with a restricted exact model.

Phase one places each carrier as a contiguous run of sub-channels. Let $\tau_i$ be the starting sub-channel of carrier~$i$ and $\ell_i = \Psi_i x_i$ the length of its run, so that carrier~$i$ occupies the half-open interval $[\tau_i, \tau_i + \ell_i)$ and an inactive carrier occupies none. The overlap of carriers $i$ and $j$ is $\mathrm{o}_{ij} = \max\{0,\ \min(\tau_i+\ell_i, \tau_j+\ell_j) - \max(\tau_i, \tau_j)\}$, which is zero whenever either carrier is inactive. Let $\mathcal{I}_F \subseteq \mathcal{I}$ be the pairs for which a single shared sub-channel already exceeds the true tolerance of one of the two carriers, $\Lambda_{ij} > \Gamma_i$ or $\Lambda_{ji} > \Gamma_j$. The excess $\delta_i$ of the interference received by carrier~$i$ over $\Gamma_i$ is capped at a fraction $\alpha = 0.3$ of $\Gamma_i$.

\begin{model}\label{mod:relax}
\begin{align}
\max \quad & \sum_{i \in \mathcal{C}} w_i\, x_i \label{eq:relax-obj}\\
\text{s.t.} \quad
& \ell_i = \Psi_i\, x_i, \quad 1 \;\le\; \tau_i \;\le\; S - \Psi_i + 1
  & \forall i \in \mathcal{C} \label{eq:relax-length}\\
& \mathrm{o}_{ij} = \max\bigl\{0,\ \min(\tau_i+\ell_i, \tau_j+\ell_j) - \max(\tau_i, \tau_j)\bigr\}
  & \forall \{i,j\} \in \mathcal{I} \label{eq:relax-overlap}\\
& \texttt{NoOverlap}\big((\tau_i, \ell_i)_{i \in \mathcal{C}_b}\big)
  & \forall b \in \mathcal{B} \label{eq:relax-beam}\\
& \texttt{NoOverlap}\big((\tau_i, \ell_i), (\tau_j, \ell_j)\big)
  & \forall \{i,j\} \in \mathcal{I}_F \label{eq:relax-forbid}\\
& \sum_{j \,:\, \{i,j\} \in \mathcal{I}} \Lambda_{ij}\, \mathrm{o}_{ij} \;\le\; \Gamma_i + \delta_i
  & \forall i \in \mathcal{C} \label{eq:relax-viol}\\
& 0 \;\le\; \delta_i \;\le\; \alpha \Gamma_i
  & \forall i \in \mathcal{C} \label{eq:relax-budget}\\
& x_i \in \{0,1\}, \quad \tau_i \in \mathbb{Z}, \quad \ell_i, \mathrm{o}_{ij} \ge 0. \notag
\end{align}
\end{model}

Constraint~\eqref{eq:relax-length} links activation to interference. An inactive carrier has an empty interval, so by~\eqref{eq:relax-overlap} it overlaps no other carrier, causes no interference to any carrier, and satisfies its own row of~\eqref{eq:relax-viol} trivially. Constraints~\eqref{eq:relax-beam} and~\eqref{eq:relax-forbid} prevent same-beam carriers and forbidden pairs from sharing a sub-channel, and~\eqref{eq:relax-viol} is the only constraint relaxed.

Phase one solves \Cref{mod:relax}. If its result violates no carrier's true tolerance, it is accepted directly. Otherwise only the violated carriers have their activation and sub-channels freed, and a restriction of \Cref{mod:direct}, as in CSA-LNS's repair (Section~\ref{sec:csa-repair}), maximises how many of them can be kept active feasibly, with every other carrier fixed where phase one placed it.

\section{Experimental Analysis}\label{sec:experiments}
This section reports how the portfolio and each of its generators perform on a common set of instances under a fixed time budget of 300 seconds.

All methods are implemented in Python and run on an Apple MacBook with an M2 chip. Exact models are solved with OR-Tools CP-SAT. CP-SAT runs with eight workers, except in relax-and-repair CP, which uses four. The versions used are OR-Tools 9.11.4210 and HiGHS (SciPy) 1.17.1. Linear relaxations and pricing subproblems are solved with HiGHS.

For every instance, the full portfolio and each generator alone are run three times, starting from the same probe-established incumbent, pool, and bound (Section~\ref{sec:probe}). We report the median weighted value over those three runs.

\subsection{Instances}\label{sec:instances}

Two instance families are taken directly from \citet{Gaudry2026}. The four realistic instances of 20, 40, 100, and 500 carriers, generated from representative user demand across the Australasia region and the HTS family comprising eight instances, one per colour of an eight-colour high-throughput GEO payload. The remaining twelve instances are synthetic, generated for this study at the same four sizes as the realistic family. Each size is generated at three interference hardness levels, sparse, moderate, and dense. Carriers are placed on a hexagonal beam layout and receive geometrically weighted randomly assigned interference levels.

Two objective types are used in the experimental analysis. The first applies uniform carrier weighting, so the objective equates to the raw active carrier count. The second assigns each carrier to one of five priority tiers with weights 16, 8, 4, 2 and 1. The carriers of each instance are shuffled with a fixed seed and assigned to the tiers in order of decreasing weight in proportions $10\%$, $32.5\%$, $32.5\%$, $15\%$ and $10\%$, so each instance always carries the same weights. Table~\ref{tab:instances} summarises the instances. Realistic instances are denoted \#-R, synthetic instances \#-S/M/D for sparse, moderate and dense interference, and the HTS instances HTS-1 to HTS-8.

\begin{table}[htbp]
\centering
\small
\begin{tabular*}{\linewidth}{@{\extracolsep{\fill}}lrrrrr@{}}
\toprule
\thead{ID} & \thead{$|\mathcal{C}|$} & \thead{$|\mathcal{B}|$} & \thead{$\sum_i \Psi_i / S$} & \thead{Median $\Psi_i$} & \thead{$r$} \\
\midrule
20-R & 20 & 6 & 3.67 & 28.0 & 0.04 \\
40-R & 40 & 8 & 4.42 & 16.0 & 1.68 \\
100-R & 100 & 17 & 9.21 & 13.5 & 0.07 \\
500-R & 500 & 95 & 55.02 & 12.0 & 0.31 \\
\midrule
20-S / M / D & 20 & 4 & 2.90 & 19.5 & 0.03 / 0.33 / 2.65 \\
40-S / M / D & 40 & 8 & 7.00 & 21.5 & 0.08 / 0.94 / 6.64 \\
100-S / M / D & 100 & 19 & 15.73 & 22.0 & 0.01 / 0.15 / 1.07 \\
500-S / M / D & 500 & 94 & 71.62 & 20.0 & 0.04 / 0.36 / 2.44 \\
\midrule
HTS-1 & 206 & 8 & 5.65 & 4.0 & 43.20 \\
HTS-2 & 170 & 6 & 4.58 & 4.0 & 10.96 \\
HTS-3 & 92 & 8 & 2.35 & 4.0 & 0.27 \\
HTS-4 & 96 & 6 & 3.18 & 5.0 & 1.35 \\
HTS-5 & 75 & 5 & 2.50 & 5.0 & 0.22 \\
HTS-6 & 123 & 5 & 3.23 & 4.0 & 1.53 \\
HTS-7 & 94 & 5 & 3.46 & 5.5 & 2.09 \\
HTS-8 & 144 & 7 & 4.41 & 5.0 & 9.48 \\
\bottomrule
\end{tabular*}
\caption{Instance characteristics. Every instance has $S = 200$ sub-channels, and the synthetic hardness levels of each size share carriers, beams and demands. The interference ratio $r$ is the median over carriers of $\sum_{j} \Lambda_{ij} / \Gamma_i$, the interference a carrier would receive relative to its tolerance if every other carrier shared one of its sub-channels.}
\label{tab:instances}
\end{table}

\subsection{Parameter settings}\label{sec:settings}

The settings in Table~\ref{tab:portfolio-budget} were fixed during development through targeted comparisons on the benchmark instances. A longer probe left less time for the generators and lowered the final result on the largest dense instance. Column generation resumed from the probe pool continues to add useful columns while its pricing is not yet exhausted, which is when rule G3 extends it. The gain from the diagnostic assembly is apparent within a few seconds, so its share was kept small to leave sufficient time for the generators. A final assembly reserve shorter than 30 s gave no improvement and lowered the worst runs on 500-R. The minimum call length excludes short calls, whose relative overrun is largest. Above $n_{\mathrm{cut}}$ column generation rarely converges and CSA-LNS continues to improve the incumbent over repeated calls, whereas below it every generator stops improving within two cycles of the schedule. A smaller stall count $h_{\mathrm{stall}}$ stops the search or triggers the warm-to-cold switch before the pool has diversified, and a larger count delays the switch beyond the budget. The per-call caps for heuristic CG and FLIP-CP were introduced because these methods overran their allowances on large instances and displaced CSA-LNS calls. Tuning stopped once further comparisons gave diminishing returns, and testing the sensitivity of these settings on a wider range of instances is left for future work.

The generator settings were chosen in the same way. CSA-LNS performs better when started from the incumbent than from a fresh construction. Before its search begins, it releases the incumbent carriers that the LP relaxation does not fully activate, which usually performed better than keeping every carrier. Releasing the same number of carriers at random performed comparably, so the gain comes from the release itself. The cost of reassigning carriers grows faster than the number reassigned, and the cap of 20 removes the resulting overrun on large instances without lowering the result. Limiting the repair pool by the degeneracy of its interference graph performed better than limiting its edge count, and matches an unlimited pool in a small fraction of the time. Repeating the exact repair after each kick added time without improvement in testing, so it runs once per call, while the exit solve occasionally finds a further gain. Ranking kick candidates by how many inactive carriers each blocks produced the largest gains, and the marginal criterion covers carriers stalled by many small contributions that no single active carrier causes. Near the activation ceiling, unlimited kicks consume the remaining budget without improvement, and the limit on consecutive unsuccessful kicks recovers that time. For targeted eviction, removing $15\%$ of the active carriers, capped at 20, outperformed both smaller and uncapped removal sizes, and scoring each carrier by the remaining tolerance of the carriers it interferes with predicted the gain of an eviction more closely than raw interference in testing. Several rules for choosing the FLIP+ starting carrier were compared, and the best, the largest ratio $\Psi_i/\Gamma_i$, was kept.

\subsection{Results}

Each instance was solved three times under each objective with a 300-second budget, and all values reported are medians over those runs. For comparison, each generator was also run alone after its own probe (Section~\ref{sec:probe}), with a 30-second allowance and a 15-second assembly, followed by the full 300-second budget, so a standalone run can take up to 45 seconds longer than the portfolio. CSA-LNS, targeted eviction and heuristic CG start from the probe incumbent, relax-and-repair CP returns it whenever its repair does not improve on it, and FLIP+ and FLIP-CP construct their solutions without it. The direct IP model was solved once per instance with CP-SAT under the same 300-second limit and without a warm start, and its reported times include building the model.

\subsubsection{Portfolio performance}

Tables~\ref{tab:portfolio-vs-generators-ip} and~\ref{tab:portfolio-vs-generators-tiered} report the portfolio's results per instance alongside the upper bound $U$. Its median objective reaches $U$, which proves optimality, on 11 of the 24 instances under the uniform objective and 8 under the priority-tiered objective. Its median objective lies within $5\%$ of $U$ on 15 and 17 instances, and the median optimality gap across instances is $2.6\%$ and $1.5\%$. Gaps above $10\%$ occur only on 100-D, 500-M and 500-D under both objectives and on HTS-1 under the uniform objective. Across the three runs, the best and median values differ on only four instances under the uniform objective and five under the priority-tiered objective, by at most $1.2\%$. The portfolio completes in $102.4$\,s and $113.2$\,s per instance on average, about a third of the 300-second budget under both objectives.

The realistic and HTS instances are shared with \citet{Gaudry2026}, so the uniform-objective results also indicate what removing the contiguity requirement gains. Against the best value reported by the contiguous methods of \citet{Gaudry2026} and \citet{Li2025}, the portfolio matches seven of the twelve instances and improves on the other five. All seven matches are proven optimal, so on these instances the contiguous solution is already optimal for the non-contiguous problem. On the realistic instances the gain is at most one carrier, while on HTS-1, HTS-8 and HTS-2 it reaches up to $7\%$, and these are the three instances with the highest interference ratio $r$ in the benchmark (Table~\ref{tab:instances}). The benefit of non-contiguous allocation is therefore concentrated on the instances where interference is densest.

\subsubsection{Generators and the virtual best solver}
\begin{table}[htbp]
\centering
\small
\setlength{\tabcolsep}{5pt}
\begin{tabular}{@{}lrrrrrrrrrr@{}}
\toprule
& & \multicolumn{2}{c}{\thead{Portfolio}} & \multicolumn{6}{c}{\thead{Generators run alone}} & \\
\cmidrule(lr){3-4} \cmidrule(lr){5-10}
\thead{ID} & \thead{$U$} & \thead{Obj.} & \thead{Time (s)} & \thead{CSA-\\LNS} & \thead{Evict.} & \thead{Relax-\\repair} & \thead{FLIP+} & \thead{Heur.\\CG} & \thead{FLIP-\\CP} & \thead{Direct\\IP} \\
\midrule
20-R & 18 & \textbf{18} & 0.10 & \textbf{18} & \textbf{18} & \textbf{18} & \textbf{18} & \textbf{18} & \textbf{18} & \textbf{18} \\
40-R & 35 & \textbf{35} & 17.4 & 34 & \textbf{35} & 34 & 23 & \textbf{35} & 32 & 32 \\
100-R & 99 & \textbf{99} & 8.60 & \textbf{99} & \textbf{99} & \textbf{99} & 96 & 98 & 98 & 95 \\
500-R & 487 & \textbf{457} & 269.3 & 453 & 453 & 448 & 448 & 448 & 447 & -- \\
\midrule
20-S & 20 & \textbf{20} & $<$0.05 & \textbf{20} & \textbf{20} & \textbf{20} & \textbf{20} & \textbf{20} & \textbf{20} & \textbf{20} \\
20-M & 18 & \textbf{17} & 1.50 & \textbf{17} & \textbf{17} & \textbf{17} & 14 & \textbf{17} & 16 & \textbf{17} \\
20-D & 15 & \textbf{15} & 0.30 & \textbf{15} & \textbf{15} & \textbf{15} & 14 & \textbf{15} & 14 & 14 \\
40-S & 38 & \textbf{37} & 37.4 & \textbf{37} & \textbf{37} & 36 & 33 & \textbf{37} & 34 & 34 \\
40-M & 30 & \textbf{27} & 40.9 & 26 & \textbf{27} & \textbf{27} & 22 & \textbf{27} & 21 & 24 \\
40-D & 22 & \textbf{20} & 24.8 & 19 & \textbf{20} & \textbf{20} & 15 & \textbf{20} & 15 & 19 \\
100-S & 100 & \textbf{100} & 23.5 & \textbf{100} & 99 & \textbf{100} & 99 & \textbf{100} & \textbf{100} & 98 \\
100-M & 89 & \textbf{83} & 161.5 & \textbf{83} & \textbf{83} & \textbf{83} & 78 & \textbf{83} & 76 & 57 \\
100-D & 82 & \textbf{64} & 144.7 & 61 & 61 & 61 & 58 & 61 & 57 & 46 \\
500-S & 493 & \textbf{480} & 270.7 & \textbf{480} & \textbf{480} & \textbf{480} & \textbf{480} & \textbf{480} & 479 & -- \\
500-M & 466 & \textbf{367} & 283.7 & 366 & 360 & 344 & 344 & 344 & 347 & -- \\
500-D & 424 & \textbf{228} & 300.1 & 226 & 206 & 188 & 188 & 189 & 191 & -- \\
\midrule
HTS-1 & 192 & \textbf{168} & 300.0 & 162 & 164 & 158 & 149 & 166 & 159 & -- \\
HTS-2 & 169 & 162 & 300.0 & 161 & 159 & 162 & 127 & \textbf{163} & 151 & -- \\
HTS-3 & 92 & \textbf{92} & 0.50 & \textbf{92} & \textbf{92} & \textbf{92} & \textbf{92} & \textbf{92} & \textbf{92} & \textbf{92} \\
HTS-4 & 94 & \textbf{94} & 3.60 & \textbf{94} & \textbf{94} & \textbf{94} & 93 & \textbf{94} & 93 & \textbf{94} \\
HTS-5 & 75 & \textbf{75} & 0.20 & \textbf{75} & \textbf{75} & \textbf{75} & \textbf{75} & \textbf{75} & \textbf{75} & \textbf{75} \\
HTS-6 & 123 & \textbf{123} & 1.50 & \textbf{123} & \textbf{123} & \textbf{123} & 120 & \textbf{123} & 121 & 119 \\
HTS-7 & 82 & \textbf{82} & 1.40 & \textbf{82} & \textbf{82} & \textbf{82} & 71 & \textbf{82} & 73 & 77 \\
HTS-8 & 131 & \textbf{127} & 265.2 & 126 & 126 & 126 & 109 & \textbf{127} & 109 & 106 \\
\midrule
\multicolumn{2}{@{}l}{Mean time (s)} & & 102.4 & 30.4 & 15.5 & 130.9 & 14.9 & 110.3 & 132.0 & 282.9 \\
\bottomrule
\end{tabular}
\caption{Uniform objective. Median objective of the portfolio, each generator run alone and the direct IP model, with a 300-second budget, and the portfolio's median runtime. Generator times include the probe of the standalone run. The portfolio and the generators are summarised over three runs and the direct IP model is a single run. The best value in each row is in bold, and \texttt{--} marks an instance the direct IP model was unable to solve within the budget.}
\label{tab:portfolio-vs-generators-ip}
\label{tab:portfolio-vs-directip-full}
\end{table}

\begin{table}[htbp]
\centering
\small
\setlength{\tabcolsep}{5pt}
\begin{tabular}{@{}lrrrrrrrrrr@{}}
\toprule
& & \multicolumn{2}{c}{\thead{Portfolio}} & \multicolumn{6}{c}{\thead{Generators run alone}} & \\
\cmidrule(lr){3-4} \cmidrule(lr){5-10}
\thead{ID} & \thead{$U$} & \thead{Obj.} & \thead{Time (s)} & \thead{CSA-\\LNS} & \thead{Evict.} & \thead{Relax-\\repair} & \thead{FLIP+} & \thead{Heur.\\CG} & \thead{FLIP-\\CP} & \thead{Direct\\IP} \\
\midrule
20-R & 114 & \textbf{114} & 0.10 & \textbf{114} & \textbf{114} & \textbf{114} & \textbf{114} & \textbf{114} & \textbf{114} & 112 \\
40-R & 220 & \textbf{211} & 22.3 & 210 & \textbf{211} & \textbf{211} & 124 & \textbf{211} & 184 & 188 \\
100-R & 590 & 588 & 67.9 & 588 & 588 & 587 & 566 & \textbf{589} & 572 & 533 \\
500-R & 2910 & \textbf{2757} & 280.6 & 2745 & 2717 & 2655 & 2655 & 2655 & 2680 & -- \\
\midrule
20-S & 120 & \textbf{120} & $<$0.05 & \textbf{120} & \textbf{120} & \textbf{120} & \textbf{120} & \textbf{120} & \textbf{120} & \textbf{120} \\
20-M & 110 & \textbf{107} & 1.80 & 105 & \textbf{107} & \textbf{107} & 85 & \textbf{107} & 101 & 102 \\
20-D & 96 & \textbf{95} & 1.10 & \textbf{95} & \textbf{95} & \textbf{95} & 77 & \textbf{95} & 81 & 85 \\
40-S & 226 & \textbf{225} & 74.3 & \textbf{225} & \textbf{225} & \textbf{225} & 205 & \textbf{225} & 213 & 214 \\
40-M & 187 & \textbf{174} & 101.7 & 170 & 170 & 170 & 148 & \textbf{174} & 140 & 156 \\
40-D & 163 & \textbf{158} & 15.3 & 150 & 154 & 156 & 108 & \textbf{158} & 124 & 90 \\
100-S & 592 & \textbf{592} & 23.2 & \textbf{592} & 588 & \textbf{592} & 588 & 588 & \textbf{592} & 588 \\
100-M & 555 & \textbf{512} & 270.9 & 506 & 505 & 494 & 472 & 504 & 464 & 409 \\
100-D & 529 & \textbf{415} & 178.0 & 414 & 410 & 407 & 386 & 410 & 364 & 227 \\
500-S & 2927 & \textbf{2830} & 271.2 & 2810 & 2798 & 2798 & 2798 & 2798 & 2810 & -- \\
500-M & 2796 & \textbf{2268} & 273.5 & 2203 & 2183 & 2012 & 2012 & 2013 & 2054 & -- \\
500-D & 2671 & \textbf{1388} & 300.1 & 1320 & 1306 & 1237 & 1127 & 1241 & 1170 & -- \\
\midrule
HTS-1 & 1185 & \textbf{1105} & 300.1 & 1079 & 1083 & 1087 & 869 & 1098 & 924 & -- \\
HTS-2 & 999 & 985 & 300.0 & 981 & 980 & 980 & 828 & \textbf{986} & 932 & -- \\
HTS-3 & 541 & \textbf{541} & 0.50 & \textbf{541} & \textbf{541} & \textbf{541} & \textbf{541} & \textbf{541} & \textbf{541} & \textbf{541} \\
HTS-4 & 568 & \textbf{568} & 3.90 & \textbf{568} & 567 & \textbf{568} & 556 & \textbf{568} & 566 & 556 \\
HTS-5 & 446 & \textbf{446} & 0.20 & \textbf{446} & \textbf{446} & \textbf{446} & \textbf{446} & \textbf{446} & \textbf{446} & \textbf{446} \\
HTS-6 & 728 & \textbf{728} & 3.20 & \textbf{728} & \textbf{728} & \textbf{728} & 719 & \textbf{728} & 720 & 718 \\
HTS-7 & 508 & \textbf{508} & 2.10 & \textbf{508} & \textbf{508} & \textbf{508} & 419 & \textbf{508} & 399 & 477 \\
HTS-8 & 809 & \textbf{796} & 223.8 & 794 & 790 & 786 & 619 & \textbf{796} & 613 & 702 \\
\midrule
\multicolumn{2}{@{}l}{Mean time (s)} & & 113.2 & 34.5 & 15.9 & 149.5 & 15.2 & 109.8 & 134.8 & 284.1 \\
\bottomrule
\end{tabular}
\caption{Priority-tiered objective. Median objective of the portfolio, each generator run alone and the direct IP model, with a 300-second budget, and the portfolio's median runtime. Generator times include the probe of the standalone run. The portfolio and the generators are summarised over three runs and the direct IP model is a single run. The best value in each row is in bold, and \texttt{--} marks an instance the direct IP model was unable to solve within the budget.}
\label{tab:portfolio-vs-generators-tiered}
\end{table}

\begin{table}[htbp]
\centering
\begin{tabular*}{\linewidth}{@{\extracolsep{\fill}}l|cc|cc@{}}
\toprule
& \multicolumn{2}{c|}{\thead{Portfolio wins / ties / losses}} & \multicolumn{2}{c}{\thead{Mean gap to best known}} \\
\thead{Method} & \thead{Uniform} & \thead{Priority-tiered} & \thead{Uniform} & \thead{Priority-tiered} \\
\midrule
Portfolio & -- & -- & $0.03\%$ & $0.01\%$ \\
CSA-LNS & 10 / 14 / 0 & 13 / 11 / 0 & $0.99\%$ & $0.97\%$ \\
Targeted eviction & 8 / 16 / 0 & 13 / 11 / 0 & $0.99\%$ & $1.00\%$ \\
Relax-and-repair CP & 8 / 16 / 0 & 12 / 12 / 0 & $1.81\%$ & $1.66\%$ \\
FLIP+ & 19 / 5 / 0 & 20 / 4 / 0 & $9.28\%$ & $11.29\%$ \\
Heuristic CG & 6 / 17 / 1 & 8 / 14 / 2 & $1.34\%$ & $1.28\%$ \\
FLIP-CP & 19 / 5 / 0 & 19 / 5 / 0 & $6.73\%$ & $8.35\%$ \\
Direct IP & 12 / 6 / 0 & 15 / 3 / 0 & $7.27\%$ & $10.17\%$ \\
\midrule
Virtual best solver & 5 / 18 / 1 & 7 / 15 / 2 & $0.33\%$ & $0.46\%$ \\
\bottomrule
\end{tabular*}
\caption{Portfolio against each generator run alone, the direct IP model and the virtual best solver, which takes the best generator median on each instance. Counts compare medians per instance, and the direct IP model is compared on the 18 instances where it returned a solution. The gap to the best known value, the largest median of any method on the instance, is averaged over instances.}
\label{tab:portfolio-vs-generators-stats}
\end{table}

Run alone, the generators show the strengths that motivate combining them (Table~\ref{tab:portfolio-vs-generators-stats}). CSA-LNS and targeted eviction come within about $1\%$ of the best known value on average, heuristic CG and relax-and-repair CP within $2\%$, and the two FLIP constructions, which trade quality for speed, within $7$ to $11\%$. The direct IP model, which has no warm start, matches the portfolio on only 6 of its 18 instances under the uniform objective and 3 under the priority-tiered objective. It was unable to solve HTS-1, HTS-2 or the 500-carrier instances within the budget.

The portfolio matches or exceeds every generator on almost every instance. It loses to a generator on three instance--objective pairs, each by one unit and at most $0.6\%$, all to heuristic CG. These are HTS-2 under both objectives, 163 against 162 and 986 against 985, and 100-R under the priority-tiered objective, 589 against 588. A one-sided sign test on the per-instance wins and losses, Holm-corrected across the fourteen comparisons, confirms the portfolio's advantage at the $5\%$ level in every comparison except heuristic CG, with six wins and one loss under the uniform objective and eight wins and two losses under the priority-tiered objective.

The virtual best solver is a stronger reference, since it selects the best generator on each instance with hindsight. The portfolio still exceeds it on 5 instances under the uniform objective and 7 under the priority-tiered objective, by up to $4.9\%$ and $5.2\%$, and falls short only on the three pairs above. Its gains come on the larger and denser instances, 100-D, 500-R, 500-M, 500-D and HTS-1 under both objectives, and also 100-M and 500-S under the priority-tiered objective. On all but one of these the returned incumbent comes mainly from a CSA-LNS call inside the portfolio, which starts from the shared incumbent after earlier calls and assemblies have had the chance to improve it, and exceeds what CSA-LNS reaches alone from the probe's incumbent. The exception is 500-R under the uniform objective, where targeted eviction sets it in two of the three runs.

Under the uniform objective, assembly sets the incumbent the portfolio returns on $47\%$ of the instances, CSA-LNS on $29\%$, FLIP+ on $21\%$ and targeted eviction on $3\%$, with each instance divided equally among its three runs. Under the priority-tiered objective, CSA-LNS sets it on $42\%$, assembly on $33\%$, FLIP+ on $17\%$ and heuristic CG on $8\%$. FLIP+ accounts for the four instances under each objective that the probe certifies optimal, and for 500-S under the uniform objective, where no generator raises it within the budget. FLIP-CP raises the incumbent only on 500-S under the priority-tiered objective, where it improves on FLIP+ in every run before CSA-LNS raises it further.

\subsubsection{Portfolio mechanisms}

Orchestration determines how much of the budget the portfolio uses. Obtaining the virtual best solver requires running all six generators, which takes $433.9$\,s per instance on average under the uniform objective and $459.8$\,s under the priority-tiered objective. The portfolio needs $102.4$\,s and $113.2$\,s, which is $4.2$ and $4.1$ times less, and it is faster on every instance. These generator times include one probe per generator run. Subtracting the full 45-second probe allowance from every generator run, and counting a run shorter than the allowance as zero, still leaves the virtual best solver $2.7$ times slower under both objectives. This overstates the probe cost, since FLIP+ run alone, which is little more than the probe, takes $14.9$\,s on average.

The saving comes from stopping early. The gate ends the run once the gap closes, which covers the proven-optimal instances, and the stall rules stop the generators once further calls stop adding value, so 15 instances under the uniform objective and 12 under the priority-tiered objective finish within a minute. The instances that use the full budget are 500-D, HTS-1 and HTS-2, under both objectives and all with open gaps.

A second comparison isolates the orchestration rules themselves. Two fixed schedules keep the generators, shared pool and closing assembly of the portfolio but remove the probe, the gate, the intermediate assemblies, the stopping rules and the per-call caps of Table~\ref{tab:portfolio-budget}. Both start from the FLIP+ incumbent and give every call the same 45-second slice until only the 30-second reserve for the closing assembly remains. Every call is warm-started from the incumbent, and CSA-LNS keeps its in-call stall detector. The static schedule cycles through CSA-LNS, targeted eviction, heuristic CG and FLIP-CP in a fixed order, and the random schedule draws the generator of each call uniformly at random from the same four, which are the generators the role schedules call.

\begin{table}[htbp]
\centering
\begin{tabular*}{\linewidth}{@{\extracolsep{\fill}}ll|c|ccc|c@{}}
\toprule
& & \thead{Mean\\time (s)} & \thead{Mean\\regret} & \thead{Max\\regret} & \thead{Regret\\$>1\%$} & \thead{Orchestrated\\W / T / L} \\
\midrule
Uniform & Orchestrated & $102.4$ & $0.10\%$ & $0.78\%$ & 0 & -- \\
 & Static & \multirow{2}{*}{$274.7$} & $0.36\%$ & $3.12\%$ & 4 & 4 / 5 / 1 \\
 & Random & & $0.39\%$ & $2.41\%$ & 5 & 2 / 7 / 1 \\
\midrule
Priority-tiered & Orchestrated & $113.2$ & $0.11\%$ & $1.07\%$ & 1 & -- \\
 & Static & \multirow{2}{*}{$277.1$} & $0.28\%$ & $2.73\%$ & 3 & 3 / 5 / 3 \\
 & Random & & $0.40\%$ & $2.99\%$ & 3 & 2 / 9 / 0 \\
\bottomrule
\end{tabular*}
\caption{The orchestrated portfolio against two fixed schedules that use the same generators, pool and closing assembly without any orchestration rule. Mean time is the mean over instances of the median runtime, and the time for the fixed schedules is their mean over both schedules. Regret is the shortfall of a median from the best median of any method on the instance, as a percentage of that value. W / T / L counts the wins, ties and losses of the orchestrated portfolio beyond run-to-run variation, on the 10 uniform and 11 priority-tiered instances where the three portfolios do not all return the same value in every run.}
\label{tab:orchestration-robustness}
\end{table}

Without orchestration, the same components use $274.7$\,s and $277.1$\,s per instance on average, so the portfolio needs on average $37\%$ and $41\%$ of their compute under the two objectives (Table~\ref{tab:orchestration-robustness}). Its solutions are also better on average. Its mean regret is $0.10\%$ against $0.36\%$ and $0.39\%$ under the uniform objective, and $0.11\%$ against $0.28\%$ and $0.40\%$ under the priority-tiered objective. Its worst case is $0.78\%$ and $1.07\%$, where the fixed schedules reach $2.4$ to $3.1\%$. On the instances where the three portfolios differ, the orchestrated portfolio wins more often than it loses in three of the four comparisons and wins and loses three each against the static schedule under the priority-tiered objective. These instances are too few for a sign test.

The three losses to the static schedule are on 500-D, 500-R and HTS-8 under the priority-tiered objective, by $1.1\%$, $1.0\%$ and $0.3\%$ of the best known value. On 500-D and 500-R the probe runs to its time limit without closing the gap, and on 500-D rule G3 then extends column generation. The generators therefore have about 90\,s and 180\,s of the budget, against about 270\,s under the static schedule, which computes no bounds. The largest wins are on 100-M and 500-M under the same objective, by $2.7\%$ and $1.9\%$. There the first CSA-LNS call of the portfolio starts from the incumbent and pool left by the probe and reaches a higher value than the first CSA-LNS call of the static schedule.

\begin{table}[htbp]
\centering
\begin{tabular*}{\linewidth}{@{\extracolsep{\fill}}lccc@{}}
\toprule
\thead{Winning bound} & \thead{Cases} & \thead{Median\\improvement} & \thead{Max\\improvement} \\
\midrule
Reduced-cost & 9 & $20.8\%$ & $45.0\%$ \\
Beam-aggregation & 28 & $5.8\%$ & $18.0\%$ \\
Minimum-overlap maximum-tolerance & 1 & $15.2\%$ & $15.2\%$ \\
Trivial & 10 & -- & -- \\
\midrule
All & 48 & $5.0\%$ & $45.0\%$ \\
\bottomrule
\end{tabular*}
\caption{Bound that sets $U$ on each instance under each objective, and its improvement on the trivial bound $\sum_{i\in\mathcal C}w_i$.}
\label{tab:bound-source}
\end{table}

The proven upper bound improves on the trivial bound on 38 of the 48 instance--objective pairs (Table~\ref{tab:bound-source}). The beam-aggregation bound sets $U$ most often, on 28 pairs, and the reduced-cost bound sets it on 9 pairs, with the largest improvements, $20.8\%$ at the median. The minimum-overlap maximum-tolerance bound sets $U$ once, on 500-D under the uniform objective. The trivial bound stands on the ten pairs where the returned incumbent activates every carrier, so no constraint binds.

\subsection{Discussion}\label{sec:exp-discussion}

The portfolio returns solutions at least as good as the best individual generator on 23 of the 24 instances under the uniform objective and 22 under the priority-tiered objective, while using roughly a quarter of the compute needed to identify that generator. The fixed schedules show that the stopping rules account for most of the saving, since the same components run to the budget without them use two to three times the compute and return slightly worse solutions on average. That comparison removes every orchestration rule at once, so it does not attribute the difference in solution quality to any single rule. The portfolio removes the need to know in advance which generator suits an instance, and it reaches the quality of the virtual best solver without the cost of running every generator. The same comparison shows what the probe costs on the largest instances. Where its bounds and columns leave the gap open, as on 500-D and 500-R under the priority-tiered objective, the probe leaves the generators less time, and the static schedule ends about $1\%$ higher. A probe allowance that adapts to the progress of column generation could return that time to the generators.

The gains over the virtual best solver show that the portfolio does more than select a generator. They appear on the large and dense instances, where no single generator dominates. On all but one of these, the returned incumbent comes mainly from a CSA-LNS call that starts from an incumbent refined by earlier calls and assemblies and exceeds what CSA-LNS reaches alone. Assembly sets $47\%$ and $33\%$ of the returned incumbents under the two objectives. These shares record which component produced each incumbent, and no experiment runs the portfolio without the shared pool, so they do not measure the contribution of the pool on its own. CSA-LNS carries most of the 100- and 500-carrier instances and the two largest HTS instances, assembly most of the smaller instances, and FLIP+ settles the easy ones almost immediately, so the orchestration can stop early wherever the probe and the first calls already close the gap.

The upper bounds make that early stopping possible. They certify optimality on 11 and 8 instances. On 5 of these under each objective every carrier is admitted and the trivial bound suffices, so the bounds of Section~\ref{sec:bounds} decide optimality on 6 and 3, and the portfolio stops as soon as it finds an optimal solution.

\section{Conclusion}\label{sec:conclusions}

We formulated non-contiguous channel allocation, showed it NP-hard by reduction from maximum independent set (Section~\ref{sec:formulation}), and developed a pattern-based reformulation, three upper-bounding methods and a cooperative portfolio of matheuristics. The portfolio and each generator were run three times on 24 instances in three families, realistic, synthetic and HTS, under a uniform and a priority-tiered objective, with the instance as the unit of analysis. The results support the following conclusions.

The portfolio matches or exceeds the best individual generator on 23 of the 24 instances under the uniform objective and 22 under the priority-tiered objective, while using roughly a quarter of the compute needed to run every generator and select the best. It exceeds the virtual best solver on 5 and 7 instances, all large or dense, and its median reaches the proven upper bound on 11 and 8 instances, of which 6 and 3 require carriers to be rejected (Section~\ref{sec:exp-discussion}).

Run alone, CSA-LNS and targeted eviction come within about $1\%$ of the best known value on average. Within the portfolio, assembly and CSA-LNS together set about three quarters of the returned incumbents under each objective, assembly more often under the uniform objective and CSA-LNS under the priority-tiered objective. The large-instance schedule gives CSA-LNS most of its calls, so its share also reflects how often it runs.

Across the 18 instances it solved, the portfolio matches or exceeds the direct model in all 36 instance--objective comparisons, with 27 wins and 9 ties (Tables~\ref{tab:portfolio-vs-generators-ip} and~\ref{tab:portfolio-vs-generators-tiered}). The difference grows with size. At 20 carriers the direct model ties the portfolio on 4 of the 8 instance--objective pairs, at 100 carriers it returns as little as $55\%$ of the portfolio's value (100-D under the priority-tiered objective, 227 against 415), and it was unable to solve any of the six instances with 170 or more carriers within the budget. The model carries $O(|\mathcal{I}|\,S)$ auxiliary variables and three times as many constraints, and it received no warm start. Model size and the weak McCormick relaxation (Section~\ref{sec:formulation}) may both contribute, and the experiments do not separate them. The pattern-based reformulation avoids these bilinear terms, and the column-generation and restricted-master methods of this paper build on it.

No single upper bound dominates, which is why the portfolio computes all three and keeps the tightest. The proven bound improves on the trivial bound on 38 of the 48 instance--objective pairs (Table~\ref{tab:bound-source}). The beam-aggregation bound sets it on 28 pairs. The reduced-cost bound sets it on 9, with the largest median improvement, $20.8\%$. It is valid for any restricted pattern set, but its correction term grows with the excess of the pricing bound over $\sigma$, so it is tight only when column generation has nearly converged.

Relaxing contiguity cannot reduce the optimal value, because every contiguous allocation remains feasible in \Cref{mod:direct} (Section~\ref{sec:formulation}). On the twelve instances shared with \citet{Gaudry2026}, the portfolio matches the best contiguous result on seven, all proven optimal, and improves on the other five, by up to $7\%$ on the instances with the densest interference.

\subsection{Future Work}

This paper addresses sub-channel allocation and carrier admission on their own, holding beam laydown, power allocation, and the illumination schedule fixed (Section~\ref{sec1}). Three extensions follow directly from that scope.

Beam-hopped systems provision capacity in the time domain through a repeating illumination schedule, held fixed throughout this work. Interference-aware channel allocation could instead be integrated with time-slice allocation, so that the two are planned jointly. Power and bandwidth allocation were taken as given inputs to this problem (Section~\ref{sec1}), jointly with each other but separately from channelisation, as is common in the literature. Embedding the methods developed here within a larger allocation loop that also revises power and bandwidth would relax that separation.

The portfolio settings were chosen through comparisons on the benchmark instances. Testing their sensitivity on a wider range of instances would show how far the settings transfer. FLIP-CP (Section~\ref{sec:flip-cp}) shows that contiguous placements can be searched within this framework. A two-stage scheme that allocates contiguously by default and falls back to non-contiguous placement only where required would trade some solution quality for simpler digital channelisation, and suits payloads where fragmentation carries a real implementation cost.

\section*{Declaration of competing interest}

\confirm{The authors declare that they have no known competing financial interests or personal relationships that could have appeared to influence the work reported in this paper.}

\section*{Funding}

\confirm{This research did not receive any specific grant from funding agencies in the public, commercial, or not-for-profit sectors.}

\section*{Data availability}

\confirm{The instances and code used in this study are available on \href{https://github.com/AngusGaudry/Non_Contiguous_Channel_Allocation_Paper}{GitHub}.}

\section*{Declaration of generative AI and AI-assisted technologies in the manuscript preparation process}

\confirm{During the preparation of this work the authors used Claude (Anthropic) in order to assist with code development and debugging, scripting the statistical analysis and tables, and language editing. After using this tool, the authors reviewed and edited the content as needed and take full responsibility for the content of the published article.}

\printcredits

\bibliographystyle{cas-model2-names}
\bibliography{cas-refs}

\end{document}